\newcommand{\abs}[1]{\left\vert#1\right\vert}
\documentclass{article}

\usepackage[top=1.5in, left=1in, right=1in, bottom=1in]{geometry}

\usepackage[utf8]{inputenc}
\usepackage{gensymb}
\usepackage{amsmath,amsthm,verbatim,amssymb,amsfonts,amscd}
\usepackage{graphicx}
\usepackage{systeme}
\usepackage{xcolor}
\usepackage{hyperref}
\usepackage{caption}
\usepackage{subcaption}
\usepackage{scrextend}
\usepackage{lineno}
\usepackage{cleveref}
\usepackage[symbol]{footmisc}
\usepackage{soul}

\newtheorem{theorem}{Theorem}[section]
\newtheorem{lemma}[theorem]{Lemma}

\newtheorem{claim}{Claim}
\newtheorem{case}{Case}

\theoremstyle{definition}
\newtheorem{definition}[theorem]{Definition}

\DeclareMathOperator{\p}{P}
\DeclareMathOperator{\Z}{Z}

\DeclareMathOperator{\Term}{Term}

\usepackage{tikz}
\usetikzlibrary{calc}
\usepackage{subcaption} 
\usepackage{float}
\usepackage{indentfirst}

\begin{document}

\title{The Zero Forcing Numbers of Peony Graphs and Web Graphs}

\author{Sara Anderton\thanks{Odessa, TX 79765 (sarajanderton@gmail.com)} \and Kanno Mizozoe\thanks{Department of Mathematics, Trinity College, Hartford, CT 06106 (kanno.mizozoe@trincoll.edu)} \and Houston Schuerger\thanks{Department of Mathematics, The University of Texas Permian Basin, Odessa, TX 79762 (schuerger\_h@utpb.edu)} \and Andrew Schwartz \thanks{Department of Mathematics, Southeast Missouri State University, Cape Girardeau, MO 63701 (aschwartz@semo.edu)}}

\maketitle

\begin{abstract}
The concept of zero forcing involves a dynamic coloring process by which blue vertices cause white vertices to become blue, with the goal of forcing the entire graph blue while choosing as few as possible vertices to be initially blue.  Past research in this area has focused on structural arguments, with approaches varying from graph substructures to the interplay between local and global graph structures.  This paper explores the use of these structural concepts when determining the zero forcing number of complex classes of graphs, specifically two infinite classes of graphs each defined on multiple parameters. 
\end{abstract}



\section{Introduction}

Graph theory, a branch of discrete mathematics, delves into the study of graphs—mathematical structures comprising dots called vertices to represent elements or locations, and line segments called edges connecting vertices which are used to depict relationships between the vertices. 
 The versatility of graphs enables their application in modeling and analyzing pairwise connections among objects, expanding their utility beyond mathematical boundaries to various real-world scenarios.
In 2008, the AIM Minimum Rank - Special Graphs Work Group \cite{aim} first introduced a graph theory concept called zero forcing.

Zero forcing is a game in graph theory that involves coloring vertices blue with the aim of using the least number of initially blue vertices. In this game, each vertex starts as either blue or white. The rule permits a blue vertex to compel its lone white neighbor to turn blue. The game concludes when there are no more white vertices left to color. 
This method of analyzing graphs has played a role in the study of the inverse eigenvalue problem for graphs where it was introduced as an upper bound on the maximum nullity of graphs \cite{aim}. Its interdisciplinary applications extend beyond graph theory, finding utility in physics as quantum control theory \cite{graphinfect} and in power grid monitoring as power domination \cite{powdom} (with the role of zero forcing evident in \cite{powdomzf}). These applications highlight the versatility and importance of this method across diverse fields.

A recent area of focus in research on the zero forcing numbers of graphs has been the exploration and introduction of new graph substructures called forts (first introduced in \cite{fort}), and the interactions these substructures have with the zero forcing number of a graph.  Another recent strategy introduced in \cite{PIP} restricts a forcing process occurring in a graph to one of its induced subgraphs.  This provides one a way to transfer information concerning zero forcing between the local structure given by the induced subgraph and the global structure of the parent graph. These concepts were introduced in the abstract setting, and thus initial research focused more on the theory than the practice of using it.  Due to this, a natural direction for additional research is to explore the usefulness of these substructures when determining the zero forcing numbers of relatively complex classes of graphs. To test this utility, this paper considers two classes of graphs defined on multiple parameters. The first graph class, peony graphs, is a new graph class defined on three parameters, and is a derivative of wheel graphs. The second graph class, web graphs, is defined on two parameters and is the result of adding pendant vertices to grid graphs. In Section 2, we determine the zero forcing number of peony graphs, making extensive use of the concept of forts.  Later, in Section 3, we determine the zero forcing number for web graphs, and in various ways take advantage of the concept of restricting forcing processes to subgraphs.  Before diving into the new results we provide the following preliminary definitions to ensure clear communication of the necessary concepts.

A graph $G$ is a collection of {\em vertices} $V(G)$ and a set $E(G)$ of pairs of vertices, called {\em edges}.  If two vertices $u$ and $v$ are members of the same edge, then we denote this by $uv \in E(G)$ and we say that $u$ is {\em adjacent} to $v$ or that $u$ and $v$ are {\em neighbors}.  The {\em degree} of a vertex $v$ denoted $\deg(v)$ is the number of neighbors that $v$ has.  A {\em pendant} vertex is a vertex with degree $1$.  If $H$ is a graph such that $V(H) \subseteq V(H)$ and $E(H) \subseteq E(G)$, then $H$ is a {\em subgraph} of $G$.  Furthermore, if $H$ is a subgraph of $G$ such that given two vertices $u,v \in V(H)$ we have that $uv \in E(H)$ if and only if $uv \in E(G)$, then $H$ is a {\em vertex-induced} subgraph of $G$.  Given $S \subseteq V(G)$, the subgraph of $G$ induced by $S$ denoted $G[S]$ is the vertex-induced subgraph of $G$ with vertex set $V(G[S])=S$.   

Two graphs $G$ and $H$ are said to be {\em isomorphic} denoted $G \cong H$ if there exists a bijective function $\sigma:V(G) \rightarrow V(H)$, called an {\em isomorphism}, such that for $u,v \in V(G)$, we have $\sigma(u)\sigma(v) \in E(H)$ if and only if $uv \in E(G)$.  A graph is said to be {\em rotationally isomorphic} if there exists an isomorphism $\sigma:V(G) \rightarrow V(G)$ such that the application of $\sigma$ can be represented geometrically as a rotation $\theta$ with $0 < \theta < 2\pi$ 
of $V(G)$ onto itself.  That is, a graph is rotationally isomorphic if an appropriate drawing of it has nontrivial rotational symmetry. 
 The graph classes we discuss in this paper, peony graphs and web graphs, are both classes of rotationally isomorphic graphs.  With this in mind, modular arithmetic can be used to streamline definitions and arguments.  Specifically, let it be understood that given a vertex $v_M \in \{u_i\}_{i=1}^n$, if either $M<1$ or $M>n$, then $v_M=u_N$ where $N \equiv M \pmod n$. Also for the compactness of the notations, given an integer $n$, define $[n]$ to denote $ \{1, 2, \dots, n\}$.

A {\em path} is a sequence of vertices $(v_1,v_2,\dots,v_k)$ such that for each $i$ with $1 \leq i \leq k-1$ we have $v_iv_{i+1} \in E(G)$.  Likewise, a {\em cycle} is a sequence of vertices $(v_1,v_2,\dots,v_k)$ such that for each $i$ with $1 \leq i \leq k-1$ we have $v_iv_{i+1} \in E(G)$ and also $v_1v_k \in E(G)$.  One can define {\em path graphs} and {\em cycle graphs} in the natural way.  Specifically, the {\em path graph} on $n$ vertices denoted $P_n$ is the graph with vertex set $V(P_n)=\{v_i\}_{i=1}^n$ and edge set $E(P_n)=\{v_iv_{i+1}\}_{i=1}^{n-1}$, and the {\em cycle graph} on $n$ vertices denoted $C_n$ is the graph with vertex set $V(C_n)=\{v_i\}_{i=1}^n$ and edge set $E(C_n)=\{v_iv_{i+1}\}_{i=1}^{n-1} \cup \{v_1v_n\}$.  A {\em path cover} of a graph $G$ is a collection of vertex-induced path subgraphs $\mathcal Q=\{Q_i\}_{i=1}^k$ such that $\{V(Q_i)\}_{i=1}^k$ is a partition of the vertex set of $G$.  The {\em path cover number} of $G$ denoted $\p(G)$ is the minimum cardinality among the path covers of $G$.  A path cover $\mathcal Q$ is a {\em minimum path cover} of $G$ if $\abs{\mathcal Q}=\p(G)$.

Zero forcing is a dynamic coloring process on graphs.  Prior to initiating the process, the vertices of a graph $G$ must first be each colored either blue or white.  Once this is done, the zero forcing process is governed by the {\em zero forcing color change rule}, which states:

\vspace{0.1in}

\begin{addmargin}[0.87cm]{0cm}

 \noindent \underline{Zero forcing color change rule:} If $u$ is blue and $v$ is the only white neighbor of $u$, then $u$ can force $v$ to be colored blue. If a vertex $u$ forces $v$, then we denote this by $u \rightarrow v$.

\end{addmargin}

\vspace{0.1in}

\noindent To initiate the process, and for the duration of the process, the zero forcing color change rule is applied until either every vertex of $G$ is blue or in its current state $G$ contains no blue vertices with a unique white neighbor.  Note that once a vertex is blue, it will remain blue.  Further note that during each application of the zero forcing color change rule, multiple vertices may become blue, and that the process may include multiple such applications, with each application forming a step in the process and these steps being referred to as {\em time-steps}.  If $B$ is the set of vertices initially colored blue and after a sufficient number of time-steps every vertex in $G$ is blue, then $B$ is said to be a {\em zero forcing set} of $G$.  The {\em zero forcing number} of $G$ denoted $\Z(G)$ is the minimum cardinality among all zero forcing sets of $G$.  A zero forcing set $B$ of $G$ is a {\em minimum zero forcing set} if $\abs{B}=\Z(G)$.

In order to be able to construct rigorous arguments for proofs, it is necessary to introduce some additional terminology for discussing the specifics of the zero forcing process.  We will be using the terminology introduced in \cite{PIP}.  The fundamental object in this discussion is the ordered collection of forces chosen during a zero forcing process called a relaxed chronology of forces.  First define $S(G,B')$ to be the set of valid forces in $G$ when $B'$ is the set of vertices which is currently blue.  A {\em relaxed chronology of forces} $\mathcal F$ is a collection of sets of forces $\{F^{(k)}\}_{k=1}^K$ such that at each time-step $k$, $F^{(k)}\subseteq S(G,E_{\mathcal F}^{[k-1]})$ where $E_{\mathcal F}^{[k-1]}$ is the set of blue vertices after time-step $k-1$, with $E_{\mathcal F}^{[0]}=B$ the set of vertices which were initially colored blue, and each vertex of $v \in V(G) \setminus B$ occurring in exactly one force $u \rightarrow v \in F^{(k)} \in \mathcal F$.  For a zero forcing set $B$ and a relaxed chronology of forces $\mathcal F$, the sequence of sets $\{E_{\mathcal F}^{[k]}\}_{k=0}^K$ with $B=E_{\mathcal F}^{[0]} \subseteq E_{\mathcal F}^{[1]} \subseteq \dots \subseteq E_{\mathcal F}^{[K-1]} \subseteq E_{\mathcal F}^{[K]}=V(G)$, is called the {\em expansion sequence} of $B$ induced by $\mathcal F$ and for each time-step $k$, $E_{\mathcal F}^{[k]}$ is called the $k$-th {\em expansion} of $B$ induced by $\mathcal F$. 

For a graph $G$, a zero forcing set $B$, and a relaxed chronology $\mathcal F$, a {\em forcing chain} induced by $\mathcal F$ is a sequence of vertices $(v_0,v_1,v_2,\dots,v_N)$ such that the vertex $v_0 \in B$, the vertex $v_N$ does not perform a force during $\mathcal F$, and for each $i$ with $0 \leq i \leq N-1$ we have $v_iv_{i+1}\in F^{(k)} \in \mathcal F$.  The collection of vertices which do not perform a force during $\mathcal F$ are the {\em terminus} of $\mathcal F$, denoted $\Term(\mathcal F)$.  The collection of forcing chains induced by $\mathcal F$ is called a {\em chain set}, and since $B$ is a zero forcing set, the forcing chains induced by $\mathcal F$ form a partition of $V(G)$.  Furthermore, since vertices can only perform a force when they have a single white neighbor, it follows that each vertex of $V(G)\setminus B$ is forced exactly once during $\mathcal F$ and each vertex of $V(G)\setminus \Term(\mathcal F)$ performs exactly one force during $\mathcal F$.  Due to this, it follows that a chain set forms a path cover, providing the following result.

\begin{theorem}{\cite{param}}\label{chainpath}
Let $G$ be a graph. Then $\p(G) \leq \Z(G)$. 
\end{theorem}    

Since forcing chains form a path cover, and paths are reversible, as a natural but very important result we have the following lemma concerning the terminus originally given in \cite{param} and generalized in \cite{PIP}.

\begin{theorem}\label{term}\cite{param,PIP}
Let $G$ be a graph, $B$ be a zero forcing set of $G$, and $\mathcal F$ be a relaxed chronology of forces of $B$ on $G$.  Then $\Term(\mathcal F)$ is a zero forcing set of $G$.     
\end{theorem}

\section{The zero forcing numbers of peony graphs}

In this section, we determine the zero forcing numbers of a class of graphs we call peony graphs.  Intuitively, a peony graph is constructed by taking the star graph $K_{1,m}$ and between each consecutive pair of pendant vertices in the star graph adding $r$ paths of length $s$. The following is the formal definition of peony graphs.

\begin{figure}[h]
    \centering
    \begin{minipage}[b]{0.45\textwidth}
        \centering
    \begin{tikzpicture}[scale=0.6]
    \coordinate (A) at (2.625, 6) {};
    \node at ($(A) + (0,0.5)$) {$u_1$}; 
    \coordinate [label=60: $u_2$] (B) at (5.25, 4.5) {};
    \coordinate [label=below right: $u_3$] (C) at (  5.25, 1.5) {};
    \coordinate (D) at (2.625, 0) {};
    \node at ($(D) - (0,0.5)$) {$u_4$}; 
    \coordinate [label=below left:$u_5$] (E) at (  0, 1.5) {};
    \coordinate [label=120: $u_6$] (F) at (  0, 4.5) {};
    \coordinate (center) at (2.625, 3) {};
    \node at ($(center) + (0.5, 0)$) {$c$}; 

    \fill[black]+(A)circle(2pt);
    \fill[black]+(B)circle(2pt);
    \fill[black]+(C)circle(2pt);
    \fill[black]+(D)circle(2pt);
    \fill[black]+(E)circle(2pt);
    \fill[black]+(F)circle(2pt);
    \fill[black]+(center)circle(2pt);
    
    \foreach \u \v in {A/D, B/E, C/F}
        \draw (\u) -- (\v);

    \foreach \u \v in {A/B, B/C, C/D, D/E, E/F, F/A}
    {
        \coordinate (1) at ($(\u)!0.2!(\v)$); 
        \coordinate (2) at ($(\u)!0.4!(\v)$); 
        \coordinate (3) at ($(\u)!0.6!(\v)$); 
        \coordinate (4) at ($(\u)!0.8!(\v)$); 
        \filldraw (1) circle (2pt);
        \filldraw (2) circle (2pt);
        \filldraw (3) circle (2pt);
        \filldraw (4) circle (2pt);
        \draw (\u) -- (\v);
    }

    \foreach \u \v in {A/B, B/C, C/D, D/E, E/F, F/A}
    {
        \draw (\u) to [bend left = 70] 
            coordinate[pos=0.2] (mid1) 
            coordinate[pos=0.4] (mid2) 
            coordinate[pos=0.6] (mid3) 
            coordinate[pos=0.8] (mid4) 
            (\v);
        \filldraw (mid1) circle (2pt);
        \filldraw (mid2) circle (2pt);
        \filldraw (mid3) circle (2pt);
        \filldraw (mid4) circle (2pt);
    }

    \draw (A) to[out= 60, in= 60, looseness=2] 
        coordinate[pos=0.2] (mid1) 
        coordinate[pos=0.4] (mid2) 
        coordinate[pos=0.6] (mid3) 
        coordinate[pos=0.8] (mid4)
        (B);
    \filldraw (mid1) circle (2pt);
    \filldraw (mid2) circle (2pt);
    \filldraw (mid3) circle (2pt);
    \filldraw (mid4) circle (2pt);
    
    \draw (B) to[out=  0, in=  0, looseness=2] 
        coordinate[pos=0.2] (mid1) 
        coordinate[pos=0.4] (mid2) 
        coordinate[pos=0.6] (mid3) 
        coordinate[pos=0.8] (mid4)
        (C);
    \filldraw (mid1) circle (2pt);
    \filldraw (mid2) circle (2pt);
    \filldraw (mid3) circle (2pt);
    \filldraw (mid4) circle (2pt);
    
    \draw (C) to[out=300, in=300, looseness=2] 
        coordinate[pos=0.2] (mid1) 
        coordinate[pos=0.4] (mid2) 
        coordinate[pos=0.6] (mid3) 
        coordinate[pos=0.8] (mid4)
        (D);
    \filldraw (mid1) circle (2pt);
    \filldraw (mid2) circle (2pt);
    \filldraw (mid3) circle (2pt);
    \filldraw (mid4) circle (2pt);
        
    \draw (D) to[out=240, in=240, looseness=2] 
        coordinate[pos=0.2] (mid1) 
        coordinate[pos=0.4] (mid2) 
        coordinate[pos=0.6] (mid3) 
        coordinate[pos=0.8] (mid4)
        (E);
    \filldraw (mid1) circle (2pt);
    \filldraw (mid2) circle (2pt);
    \filldraw (mid3) circle (2pt);
    \filldraw (mid4) circle (2pt);

    \draw (E) to[out=180, in=180, looseness=2] 
        coordinate[pos=0.2] (mid1) 
        coordinate[pos=0.4] (mid2) 
        coordinate[pos=0.6] (mid3) 
        coordinate[pos=0.8] (mid4)(F);
    \filldraw (mid1) circle (2pt);
    \filldraw (mid2) circle (2pt);
    \filldraw (mid3) circle (2pt);
    \filldraw (mid4) circle (2pt);
        
    \draw (F) to[out=120, in=120, looseness=2] 
        coordinate[pos=0.2] (mid1) 
        coordinate[pos=0.4] (mid2) 
        coordinate[pos=0.6] (mid3) 
        coordinate[pos=0.8] (mid4)
        (A);
    \filldraw (mid1) circle (2pt);
    \filldraw (mid2) circle (2pt);
    \filldraw (mid3) circle (2pt);
    \filldraw (mid4) circle (2pt);
    \end{tikzpicture}
    \caption{The peony graph $Py(6, 3, 4)$}
    \end{minipage}
    \hfill
    \begin{minipage}[b]{0.45\textwidth}
        \centering
                \begin{tikzpicture}[scale=0.6]
            \coordinate (A) at (5.25, 12) {};
            \node at ($(A) + (0,0.5)$) {$u_1$}; 
            \coordinate [label=60: $u_2$] (B) at (  10.5, 9) {};
            \coordinate (center) at (5.25, 6) {};
            \node at ($(center) + (-0.2,0)$) {$c$};
            \coordinate [label=$ $](null) at (4, 5){};

            \fill[black]+(A)circle(3pt);
            \fill[black]+(B)circle(2pt);
            \fill[black]+(center)circle(2pt);
    
            \foreach \u \v in {A/center, B/center}
                \draw (\u) -- (\v);

            \foreach \u \v in {A/B}
            {
                \coordinate [label=$v_{1, 1, 1}$](1) at ($(\u)!0.2!(\v)$); 
                \coordinate [label=$v_{1, 1, 2}$](2) at ($(\u)!0.4!(\v)$); 
                \coordinate [label=$v_{1, 1, 3}$](3) at ($(\u)!0.6!(\v)$); 
                \coordinate [label=$v_{1, 1, 4}$](4) at ($(\u)!0.8!(\v)$);
                \filldraw (1) circle (3pt); 
                \filldraw (2) circle (3pt);
                \filldraw (3) circle (3pt);
                \filldraw (4) circle (3pt);
                \draw (\u) -- (\v);
                \draw[ultra thick] (\u) -- (4);
            }

            \foreach \u \v in {A/B}
            {
                \draw (\u) to [bend left = 70] 
                    coordinate[pos=0.2] (mid1) 
                    coordinate[pos=0.4] (mid2) 
                    coordinate[pos=0.6] (mid3) 
                    coordinate[pos=0.8] (mid4) 
                    (\v);
                \filldraw (mid1) circle (3pt);
                \filldraw (mid2) circle (3pt);
                \filldraw (mid3) circle (3pt);
                \filldraw (mid4) circle (3pt);
                \node at (mid1) [above right] {$v_{1, 2, 1}$};
                \node at (mid2) [above right] {$v_{1, 2, 2}$};
                \node at (mid3) [above right] {$v_{1, 2, 3}$};
                \node at ($(mid4) + (0.15, 0.5)$) {$v_{1, 2, 4}$};
                \draw[ultra thick] (\u) to [bend left = 10] (mid1);
                \draw[ultra thick] (mid1) to [bend left = 30] (mid4);
            }
            \foreach \u \v in {A/B}
            {
                \draw (\u) to[out= 60, in= 60, looseness=2] 
                    coordinate[pos=0.2] (mid1) 
                    coordinate[pos=0.4] (mid2) 
                    coordinate[pos=0.6] (mid3) 
                    coordinate[pos=0.8] (mid4)
                    (\v);
                \filldraw (mid1) circle (3pt);
                \filldraw (mid2) circle (3pt);
                \filldraw (mid3) circle (3pt);
                \filldraw (mid4) circle (3pt);
                \node at (mid1) [above] {$v_{1, 3, 1}$};
                \node at (mid2) [above right] {$v_{1, 3, 2}$};
                \node at (mid3) [above right] {$v_{1, 3, 3}$};
                \node at (mid4) [above right] {$v_{1, 3, 4}$};
                \draw[ultra thick] (\u) to [bend left = 18] (mid1);
                \draw[ultra thick] (mid1) to [bend left = 18] (mid2);
                \draw[ultra thick] (mid2) to [bend left = 18] (mid3);
                \draw[ultra thick] (mid3) to [bend left = 18] (mid4);
            }
        \end{tikzpicture}
        \caption{A station $S_1$ of the peony graph $Py(6, 3, 4)$}
    \end{minipage}
\end{figure}

\begin{definition}[Peony Graph]
A {\em peony graph}, denoted as $Py(m, r, s)$, is a graph with a vertex set $V(Py(m, r, s)) = \{c\} \cup \{u_i\}_{i = 1} ^ {m}\cup \{v_{i, j, k}\}_{i = 1,}^m{} _{j=1,}^r{} _{k=1}^s{}$.
Given $w, z \in V(Py(m, r, s))$ distinct, it follows that $wz \in E(Py(m, r, s))$ if one of the following is true.

    \begin{itemize}
     \item $\{w, z\} = \{c, u_i\}$ for some $i \in \{1, 2, ..., m\}$,
     \item $\{w, z\} = \{u_i, v_{i,j,1}\}$ for some  $i \in \{1, 2, \dots, m\}$ and $j \in \{1, 2, \dots r\}$, 
     \item $\{w, z\} = \{u_i, v_{i-1, j, s}\}$ for some $i \in \{1, 2, \dots, m\}$  and $j \in \{1, 2, \dots, r\}$.
     \item $\{w, z\} = \{v_{i, j, k}, v_{i, j, k+1}\}$ for some $i \in \{1,     2, \dots, m\}, j \in \{1, 2, \dots, r\}$, and $k \in \{1, 2, \dots, s - 1\}$.
    \end{itemize}
    
\end{definition}

To facilitate our proofs, it is useful to define certain substructures of peony graphs. For every $i \in \{1, 2, \dots, m\}$, the $i$-th station, denoted $\mathcal S_i$, is the collection of vertices $\{u_i\} \cup \{v_{i,j,k}\}_{j=1,}^r{}_{k=1}^s$. For each $i,j$ with $i \in \{1, 2, \dots, m\}$ and $j \in \{1, 2, \dots, r\}$, the $i,j$-th layer, denoted $S_{i,j}$, is the collection of vertices $\{v_{i,j,k}\}_{k=1}^s$.

The concept of forts were introduced in \cite{fort} along with the following theorem.  Forts can be considered as obstacles to the zero forcing processes. So they are often studied alongside zero forcing. For the purpose of this paper, forts will help us to determine the zero forcing number of peony graphs.

\begin{definition}\cite{fort}\label{def:fort}\normalfont
Let $G$ be a graph and $S \subseteq V(G)$. If for all $u \in V(G)\backslash S$, $|N_{G}(u) \cap S| \neq 1$, then we refer to $S$ as a {\em fort} of $G$.
\end{definition}

\begin{theorem}{\normalfont \cite{fort}}\label{thm:fort}
Let $G$ be a graph and $S \subseteq V(G)$.  $S$ is a zero forcing set of $G$ if and only if for each fort $F$ of $G$, $S \cap F \neq \emptyset$.
\end{theorem}

We now establish several families of forts that are present in peony graphs. We do so in order to calculate a lower bound on the zero forcing numbers of peony graphs.

\begin{figure}[h]
    \centering
    
    \begin{minipage}[b]{0.4\textwidth} 
        \centering
           \begin{tikzpicture}[scale=0.7]
            \coordinate (A) at (4.2, 9.6) {};
            \node at ($(A) + (0, 0.5)$) {$u_1$}; 
            \coordinate [label=60: $u_2$] (B) at (8.4, 7.2) {};
            \coordinate (center) at (4.2, 4.8) {};
            \node at ($(center) + (-0.2,0)$) {$c$};
            \coordinate [label=$ $](null) at (3.2, 4){};
            \fill[black]+(A)circle(2pt);
            \fill[black]+(B)circle(2pt);
            \fill[black]+(center)circle(2pt);
            \foreach \u \v in {A/center, B/center}
                \draw (\u) -- (\v);

            \foreach \u \v in {A/B}
            {
                \coordinate (1) at ($(\u)!0.2!(\v)$); 
                \coordinate (2) at ($(\u)!0.4!(\v)$); 
                \coordinate (3) at ($(\u)!0.6!(\v)$);
                \coordinate (4) at ($(\u)!0.8!(\v)$); 
                \filldraw (1) circle (2.8pt);
                \filldraw (2) circle (2.8pt);
                \filldraw (3) circle (2.8pt);
                \filldraw (4) circle (2.8pt);
                \draw (\u) -- (\v);
                \draw[ultra thick] (1) -- (4);
            }

            \foreach \u \v in {A/B}
            {
                \draw (\u) to [bend left = 70] 
                    coordinate[pos=0.2] (mid1) 
                    coordinate[pos=0.4] (mid2) 
                    coordinate[pos=0.6] (mid3) 
                    coordinate[pos=0.8] (mid4) 
                    (\v);

                \filldraw (mid1) circle (2pt);
                \filldraw (mid2) circle (2pt);
                \filldraw (mid3) circle (2pt);
                \filldraw (mid4) circle (2pt);
            }

            \foreach \u \v in {A/B}
            {
                \draw (\u) to[out= 60, in= 60, looseness=2] 
                    coordinate[pos=0.2] (mid1) 
                    coordinate[pos=0.4] (mid2) 
                    coordinate[pos=0.6] (mid3) 
                    coordinate[pos=0.8] (mid4)
                    (\v);
                \filldraw (mid1) circle (2.8pt);
                \filldraw (mid2) circle (2.8pt);
                \filldraw (mid3) circle (2.8pt);
                \filldraw (mid4) circle (2.8pt);
                \draw[ultra thick] (mid1) to [bend left = 18] (mid2);
                \draw[ultra thick] (mid2) to [bend left = 18] (mid3);
                \draw[ultra thick] (mid3) to [bend left = 18] (mid4);
            }
        \end{tikzpicture}
        \caption{Claim 1 fort}
        \label{fig:third}
    \end{minipage}
    \hfill
    \begin{minipage}[b]{0.5\textwidth} 
        \centering
        \begin{tikzpicture}[scale=0.7]
           
                \coordinate (A) at  (2.275, 5.2) {};
                \node at ($(A) + (0, 0.5)$) {$u_1$};
                \coordinate (B) at  (4.55, 3.9) {};
                \node at ($(B) + (0.5, 0.2)$) {$u_2$};
                \coordinate (C) at  (4.55, 1.3) {};
                \node at ($(C) + (0.5, -0.2)$) {$u_3$};
                \coordinate (D) at  (2.275, 0) {};
                \node at ($(D) + (0, -0.5)$) {$u_4$};
                \coordinate (E) at  (0, 1.3) {};
                \node at ($(E) + (-0.5, -0.2)$) {$u_5$};
                \coordinate [label=120: $u_6$] (F) at (  0, 3.9) {};
                \coordinate (center) at  (2.275, 2.6) {};
                \node at ($(center) + (0.5, 0)$) {$c$}; 

                \fill[black]+(A)circle(1.7pt);
                \fill[black]+(B)circle(1.7pt);
                \fill[black]+(C)circle(1.7pt);
                \fill[black]+(D)circle(1.7pt);
                \fill[black]+(E)circle(1.7pt);
                \fill[black]+(F)circle(1.7pt);
                \fill[black]+(center)circle(2.0pt);
    
                \foreach \u \v in {A/D, B/E, C/F}
                \draw (\u) -- (\v);

                \foreach \u \v in {A/B, D/E}{
                    \coordinate (1) at ($(\u)!0.2!(\v)$); 
                    \coordinate (2) at ($(\u)!0.4!(\v)$); 
                    \coordinate (3) at ($(\u)!0.6!(\v)$);
                    \coordinate (4) at ($(\u)!0.8!(\v)$); 
                    \filldraw (1) circle (1.7pt);
                    \filldraw (2) circle (1.7pt);
                    \filldraw (3) circle (1.7pt);
                    \filldraw (4) circle (1.7pt);
                    \draw (\u) -- (\v);
                    \draw[ultra thick] (1) -- (4);
                }
                \foreach \u \v in {B/C, C/D, E/F, F/A}{
                    \coordinate (1) at ($(\u)!0.2!(\v)$); 
                    \coordinate (2) at ($(\u)!0.4!(\v)$); 
                    \coordinate (3) at ($(\u)!0.6!(\v)$);
                    \coordinate (4) at ($(\u)!0.8!(\v)$); 
                    \filldraw (1) circle (1pt);
                    \filldraw (2) circle (1pt);
                    \filldraw (3) circle (1pt);
                    \filldraw (4) circle (1pt);
                    \draw (\u) -- (\v);
                }
    
    \foreach \u \v in {A/B, B/C, D/E, E/F, F/A}{
        \draw (\u) to [bend left = 70] 
            coordinate[pos=0.2] (mid1) 
            coordinate[pos=0.4] (mid2) 
            coordinate[pos=0.6] (mid3) 
            coordinate[pos=0.8] (mid4) 
            (\v);

        \filldraw (mid1) circle (1pt);
        \filldraw (mid2) circle (1pt);
        \filldraw (mid3) circle (1pt);
        \filldraw (mid4) circle (1pt);
    }
    
    \foreach \u \v in {C/D, E/F}
    {
        \draw (\u) to [bend left = 70] 
            coordinate[pos=0.2] (mid1) 
            coordinate[pos=0.4] (mid2) 
            coordinate[pos=0.6] (mid3) 
            coordinate[pos=0.8] (mid4) 
            (\v);

        \filldraw (mid1) circle (1.7pt);
        \filldraw (mid2) circle (1.7pt);
        \filldraw (mid3) circle (1.7pt);
        \filldraw (mid4) circle (1.7pt);
        \draw[ultra thick] (mid1) to [bend left = 30] (mid4);
    }

    \draw (A) to[out= 60, in= 60, looseness=2] 
        coordinate[pos=0.2] (mid1) 
        coordinate[pos=0.4] (mid2) 
        coordinate[pos=0.6] (mid3) 
        coordinate[pos=0.8] (mid4)
        (B);
    \filldraw (mid1) circle (1pt);
    \filldraw (mid2) circle (1pt);
    \filldraw (mid3) circle (1pt);
    \filldraw (mid4) circle (1pt);
    
    \draw (B) to[out=  0, in=  0, looseness=2] 
        coordinate[pos=0.2] (mid1) 
        coordinate[pos=0.4] (mid2) 
        coordinate[pos=0.6] (mid3) 
        coordinate[pos=0.8] (mid4)
        (C);
    \filldraw (mid1) circle (1.7pt);
    \filldraw (mid2) circle (1.7pt);
    \filldraw (mid3) circle (1.7pt);
    \filldraw (mid4) circle (1.7pt);
    \draw[ultra thick] (mid1) to [bend left = 60] (mid4);
    
    \draw (C) to[out=300, in=300, looseness=2] 
        coordinate[pos=0.2] (mid1) 
        coordinate[pos=0.4] (mid2) 
        coordinate[pos=0.6] (mid3) 
        coordinate[pos=0.8] (mid4)
        (D);
    \filldraw (mid1) circle (1pt);
    \filldraw (mid2) circle (1pt);
    \filldraw (mid3) circle (1pt);
    \filldraw (mid4) circle (1pt);
        
    \draw (D) to[out=240, in=240, looseness=2] 
        coordinate[pos=0.2] (mid1) 
        coordinate[pos=0.4] (mid2) 
        coordinate[pos=0.6] (mid3) 
        coordinate[pos=0.8] (mid4)
        (E);
    \filldraw (mid1) circle (1pt);
    \filldraw (mid2) circle (1pt);
    \filldraw (mid3) circle (1pt);
    \filldraw (mid4) circle (1pt);
        
    \draw (E) to[out=180, in=180, looseness=2] 
        coordinate[pos=0.2] (mid1) 
        coordinate[pos=0.4] (mid2) 
        coordinate[pos=0.6] (mid3) 
        coordinate[pos=0.8] (mid4)
        (F);
    \filldraw (mid1) circle (1pt);
    \filldraw (mid2) circle (1pt);
    \filldraw (mid3) circle (1pt);
    \filldraw (mid4) circle (1pt);
        
    \draw (F) to[out=120, in=120, looseness=2] 
        coordinate[pos=0.2] (mid1) 
        coordinate[pos=0.4] (mid2) 
        coordinate[pos=0.6] (mid3) 
        coordinate[pos=0.8] (mid4)
        (A);
    \filldraw (mid1) circle (1.7pt);
    \filldraw (mid2) circle (1.7pt);
    \filldraw (mid3) circle (1.7pt);
    \filldraw (mid4) circle (1.7pt);
    \draw[ultra thick] (mid1) to [bend left = 60] (mid4);
        \end{tikzpicture}
        \caption{Claim 2 fort}
        \label{fig:fourth}
    \end{minipage}
\end{figure}


\begin{claim}\label{peony:claim1}
For each $i \in \{1, 2, ..., m\}$ and each distinct $j_1, j_2 \in \{1, 2, \dots, r\}$, $S_{i,j_1} \cup S_{i,j_2}$ is a fort of $Py(m, r, s)$. 
\end{claim}

\begin{proof}
Let $w \in V(Py(m, r, s)) \backslash (S_{i, j_1} \cup S_{i, j_2})$. If $w \notin \{u_i, u_{i+1}\}$, then $|N_{Py(m, r, s)}(w) \cap (S_{i,j_1} \cup S_{i,j_2})| = 0$. If $w \in \{u_i, u_{i+1}\}$, then $|N_{Py(m, r, s)}(w) \cap (S_{i,j_1} \cup S_{i,j_2})| = 2$. Therefore, $S_{i,j_1} \cup S_{i,j_2}$ is a fort.
\end{proof}


\begin{claim}\label{peony:claim2}
For every $i \in \{1, 2, \dots, m\}$, let $j_i \in \{1, 2, \dots r\}$. Then $S = \bigcup \limits_{i = 1}^m{} S_i{,}_{j_i}$ is a fort of $Py(m, r, s)$.
\end{claim}

\begin{proof}
Let $w \in V(Py(m, r, s)) \backslash S$. If $w \notin \{u_i\}_{i=1}^m$, then $|N_{Py(m, r, s)}(w) \cap S| = 0$. If $w \in \{u_i\}_{i=1}^m$, then $|N_{Py(m, r, s)}(w) \cap S| = 2$, because for any $i \in \{1, 2, \dots m\}, |N_{Py(m, r, s)}(u_i) \cap S| = |N_{Py(m, r, s)}(u_i) \cap (S_{i, k_i} \cup S_{i+1, k_{i+1}})| = 2$. Therefore, $S$ is a fort.
\end{proof}

\begin{figure}[h]
    \centering
    \begin{minipage}[b]{0.4\textwidth} 
        \centering
        \begin{tikzpicture}[scale=0.7]

    \coordinate (A) at  (2.275, 5.2) {};
    \node at ($(A) + (0, 0.5)$) {$u_1$};
    \coordinate [label=60: $u_2$] (B) at (4.55, 3.9) {};
    \coordinate [label=below right: $u_3$] (C) at (  4.55, 1.3) {};
    \coordinate (D) at  (2.275, 0) {};
    \node at ($(D) + (0, -0.5)$) {$u_4$};
    \coordinate (E) at  (0, 1.3) {};
    \node at ($(E) + (-0.5, -0.2)$) {$u_5$};
    \coordinate [label=120: $u_6$] (F) at (  0, 3.9) {};
    \coordinate (center) at  (2.275, 2.6) {};
    \node at ($(center) + (0.5, 0)$) {$c$};
    \coordinate [label=$ $] (null) at (2.275, -2.1) {};

    \fill[black]+(A)circle(1.7pt);
    \fill[black]+(B)circle(1.7pt);
    \fill[black]+(C)circle(1.7pt);
    \fill[black]+(D)circle(1.7pt);
    \fill[black]+(E)circle(1.7pt);
    \fill[black]+(F)circle(1.7pt);
    \fill[black]+(center)circle(2.8pt);
    
    \foreach \u \v in {A/D, B/E, C/F}
        \draw (\u) -- (\v);

    \foreach \u \v in {A/B, D/E, E/F}
    {
        \coordinate (1) at ($(\u)!0.2!(\v)$); 
        \coordinate (2) at ($(\u)!0.4!(\v)$); 
        \coordinate (3) at ($(\u)!0.6!(\v)$);
        \coordinate (4) at ($(\u)!0.8!(\v)$); 
        \filldraw (1) circle (1.7pt);
        \filldraw (2) circle (1.7pt);
        \filldraw (3) circle (1.7pt);
        \filldraw (4) circle (1.7pt);
        \draw (\u) -- (\v);
        \draw[ultra thick] (1) -- (4);
    }
    \foreach \u \v in {B/C, C/D, E/F, F/A}
    {
        \coordinate (1) at ($(\u)!0.2!(\v)$); 
        \coordinate (2) at ($(\u)!0.4!(\v)$); 
        \coordinate (3) at ($(\u)!0.6!(\v)$);
        \coordinate (4) at ($(\u)!0.8!(\v)$); 
        \filldraw (1) circle (1pt);
        \filldraw (2) circle (1pt);
        \filldraw (3) circle (1pt);
        \filldraw (4) circle (1pt);
        \draw (\u) -- (\v);
    }

    \foreach \u \v in {A/B, B/C, D/E, E/F, F/A}
    {
        \draw (\u) to [bend left = 70] coordinate[pos=0.2] (mid1) coordinate[pos=0.4] (mid2) coordinate[pos=0.6] (mid3) coordinate[pos=0.8] (mid4) (\v);

        \filldraw (mid1) circle (1pt);
        \filldraw (mid2) circle (1pt);
        \filldraw (mid3) circle (1pt);
        \filldraw (mid4) circle (1pt);
    }
    \foreach \u \v in {C/D}
    {
        \draw (\u) to [bend left = 70] coordinate[pos=0.2] (mid1) coordinate[pos=0.4] (mid2) coordinate[pos=0.6] (mid3) coordinate[pos=0.8] (mid4) (\v);

        \filldraw (mid1) circle (1.7pt);
        \filldraw (mid2) circle (1.7pt);
        \filldraw (mid3) circle (1.7pt);
        \filldraw (mid4) circle (1.7pt);
        \draw[ultra thick] (mid1) to [bend left = 30] (mid4);
    }

    \draw (A) to[out= 60, in= 60, looseness=2] coordinate[pos=0.2] (mid1) coordinate[pos=0.4] (mid2) coordinate[pos=0.6] (mid3) coordinate[pos=0.8] (mid4)(B);
    \filldraw (mid1) circle (1pt);
    \filldraw (mid2) circle (1pt);
    \filldraw (mid3) circle (1pt);
    \filldraw (mid4) circle (1pt);
    
    \draw (B) to[out=  0, in=  0, looseness=2] coordinate[pos=0.2] (mid1) coordinate[pos=0.4] (mid2) coordinate[pos=0.6] (mid3) coordinate[pos=0.8] (mid4)(C);
    \filldraw (mid1) circle (1.7pt);
    \filldraw (mid2) circle (1.7pt);
    \filldraw (mid3) circle (1.7pt);
    \filldraw (mid4) circle (1.7pt);
    \draw[ultra thick] (mid1) to [bend left = 60] (mid4);
    
    \draw (C) to[out=300, in=300, looseness=2] coordinate[pos=0.2] (mid1) coordinate[pos=0.4] (mid2) coordinate[pos=0.6] (mid3) coordinate[pos=0.8] (mid4)(D);
    \filldraw (mid1) circle (1pt);
    \filldraw (mid2) circle (1pt);
    \filldraw (mid3) circle (1pt);
    \filldraw (mid4) circle (1pt);
        
    \draw (D) to[out=240, in=240, looseness=2] coordinate[pos=0.2] (mid1) coordinate[pos=0.4] (mid2) coordinate[pos=0.6] (mid3) coordinate[pos=0.8] (mid4)(E);
    \filldraw (mid1) circle (1pt);
    \filldraw (mid2) circle (1pt);
    \filldraw (mid3) circle (1pt);
    \filldraw (mid4) circle (1pt);
        
    \draw (E) to[out=180, in=180, looseness=2] coordinate[pos=0.2] (mid1) coordinate[pos=0.4] (mid2) coordinate[pos=0.6] (mid3) coordinate[pos=0.8] (mid4)(F);
    \filldraw (mid1) circle (1pt);
    \filldraw (mid2) circle (1pt);
    \filldraw (mid3) circle (1pt);
    \filldraw (mid4) circle (1pt);
        
    \draw (F) to[out=120, in=120, looseness=2] coordinate[pos=0.2] (mid1) coordinate[pos=0.4] (mid2) coordinate[pos=0.6] (mid3) coordinate[pos=0.8] (mid4)(A);
    \filldraw (mid1) circle (1pt);
    \filldraw (mid2) circle (1pt);
    \filldraw (mid3) circle (1pt);
    \filldraw (mid4) circle (1pt);

        \end{tikzpicture}
        \caption{Claim 3 fort}
        \label{fig:five}
    \end{minipage}
    \hfill
    \begin{minipage}[b]{0.45\textwidth}
        \centering
        \begin{tikzpicture}[scale=0.7]
    
    \coordinate (A) at  (2.275, 5.2) {};
    \node at ($(A) + (0, 0.5)$) {$u_1$};
    \coordinate [label=60: $u_2$] (B) at (4.55, 3.9) {};
    \coordinate [label=below right: $u_3$] (C) at (  4.55, 1.3) {};
    \coordinate (D) at  (2.275, 0) {};
    \node at ($(D) + (0, -0.5)$) {$u_4$};
    \coordinate (E) at  (0, 1.3) {};
    \node at ($(E) + (-0.5, -0.2)$) {$u_5$};
    \coordinate [label=120: $u_6$] (F) at (  0, 3.9) {};
    \coordinate (center) at  (2.275, 2.6) {};
    \node at ($(center) + (0.5, 0)$) {$c$};

    \fill[black]+(A)circle(1.7pt);
    \fill[black]+(B)circle(1.7pt);
    \fill[black]+(C)circle(1.7pt);
    \fill[black]+(D)circle(1.7pt);
    \fill[black]+(E)circle(1.7pt);
    \fill[black]+(F)circle(1.7pt);
    \fill[black]+(center)circle(3pt);
    
    \foreach \u \v in {A/D, B/E, C/F}
        \draw (\u) -- (\v);

    \foreach \u \v in {A/B, B/C, C/D}
    {
        \coordinate (1) at ($(\u)!0.2!(\v)$); 
        \coordinate (2) at ($(\u)!0.4!(\v)$); 
        \coordinate (3) at ($(\u)!0.6!(\v)$);
        \coordinate (4) at ($(\u)!0.8!(\v)$); 
        \filldraw (1) circle (1pt);
        \filldraw (2) circle (1pt);
        \filldraw (3) circle (1pt);
        \filldraw (4) circle (2.5pt);
        \draw (\u) -- (\v);
    }

    \foreach \u \v in {D/E, E/F, F/A}
    {
        \coordinate (1) at ($(\u)!0.2!(\v)$); 
        \coordinate (2) at ($(\u)!0.4!(\v)$); 
        \coordinate (3) at ($(\u)!0.6!(\v)$);
        \coordinate (4) at ($(\u)!0.8!(\v)$); 
        \filldraw (1) circle (1pt);
        \filldraw (2) circle (2.5pt);
        \filldraw (3) circle (1pt);
        \filldraw (4) circle (1pt);
        \draw (\u) -- (\v);
    }


    \foreach \u \v in {B/C, D/E}
    {
        \draw (\u) to [bend left = 70] coordinate[pos=0.2] (mid1) coordinate[pos=0.4] (mid2) coordinate[pos=0.6] (mid3) coordinate[pos=0.8] (mid4) (\v);

        \filldraw (mid1) circle (1pt);
        \filldraw (mid2) circle (1pt);
        \filldraw (mid3) circle (1pt);
        \filldraw (mid4) circle (2.5pt);
    }
    \foreach \u \v in {F/A}
    {
        \draw (\u) to [bend left = 70] coordinate[pos=0.2] (mid1) coordinate[pos=0.4] (mid2) coordinate[pos=0.6] (mid3) coordinate[pos=0.8] (mid4) (\v);

        \filldraw (mid1) circle (1pt);
        \filldraw (mid2) circle (1pt);
        \filldraw (mid3) circle (2.5pt);
        \filldraw (mid4) circle (1pt);
    }
    \foreach \u \v in {A/B,  C/D,  E/F}
    {
        \draw (\u) to [bend left = 70] coordinate[pos=0.2] (mid1) coordinate[pos=0.4] (mid2) coordinate[pos=0.6] (mid3) coordinate[pos=0.8] (mid4) (\v);

        \filldraw (mid1) circle (2.5 pt);
        \filldraw (mid2) circle (1pt);
        \filldraw (mid3) circle (1pt);
        \filldraw (mid4) circle (1pt);
    }

    \draw (A) to[out= 60, in= 60, looseness=2] coordinate[pos=0.2] (mid1) coordinate[pos=0.4] (mid2) coordinate[pos=0.6] (mid3) coordinate[pos=0.8] (mid4)(B);
    \filldraw (mid1) circle (2.5pt);
    \filldraw (mid2) circle (1pt);
    \filldraw (mid3) circle (1pt);
    \filldraw (mid4) circle (1pt);
    
    \draw (B) to[out=  0, in=  0, looseness=2] coordinate[pos=0.2] (mid1) coordinate[pos=0.4] (mid2) coordinate[pos=0.6] (mid3) coordinate[pos=0.8] (mid4)(C);
    \filldraw (mid1) circle (1pt);
    \filldraw (mid2) circle (2.5pt);
    \filldraw (mid3) circle (1pt);
    \filldraw (mid4) circle (1pt);
    
    \draw (C) to[out=300, in=300, looseness=2] coordinate[pos=0.2] (mid1) coordinate[pos=0.4] (mid2) coordinate[pos=0.6] (mid3) coordinate[pos=0.8] (mid4)(D);
    \filldraw (mid1) circle (1pt);
    \filldraw (mid2) circle (2.5pt);
    \filldraw (mid3) circle (1pt);
    \filldraw (mid4) circle (1pt);
        
    \draw (D) to[out=240, in=240, looseness=2] coordinate[pos=0.2] (mid1) coordinate[pos=0.4] (mid2) coordinate[pos=0.6] (mid3) coordinate[pos=0.8] (mid4)(E);
    \filldraw (mid1) circle (1pt);
    \filldraw (mid2) circle (1pt);
    \filldraw (mid3) circle (2.5pt);
    \filldraw (mid4) circle (1pt);
        
    \draw (E) to[out=180, in=180, looseness=2] coordinate[pos=0.2] (mid1) coordinate[pos=0.4] (mid2) coordinate[pos=0.6] (mid3) coordinate[pos=0.8] (mid4)(F);
    \filldraw (mid1) circle (1pt);
    \filldraw (mid2) circle (1pt);
    \filldraw (mid3) circle (2.5pt);
    \filldraw (mid4) circle (1pt);
        
    \draw (F) to[out=120, in=120, looseness=2] coordinate[pos=0.2] (mid1) coordinate[pos=0.4] (mid2) coordinate[pos=0.6] (mid3) coordinate[pos=0.8] (mid4)(A);
    \filldraw (mid1) circle (1pt);
    \filldraw (mid2) circle (1pt);
    \filldraw (mid3) circle (1pt);
    \filldraw (mid4) circle (2.5pt);
        \end{tikzpicture}
    \caption{One example of the set $B$ constructed in Claim 4}
    \label{fig:six}
    \end{minipage}
\end{figure}


\begin{claim}\label{peony:claim3}
Let $i_0 \in [m]$.
For each $i \in \{1, 2, \dots, m\} \backslash \{i_0{}\}$ choose $j_i \in \{1, 2, \dots,  r\}$. Then $S = \left(\bigcup \limits_{i \in [m] \backslash \{i_0{}\}}S_{i,j_i} \cup \{c\}\right)$ is a fort of $Py(m, r, s)$.
\end{claim}

\begin{proof}
Let $w \in V(Py(m, r, s)) \backslash S$. If $w \notin \{u_i\}_{i=1}^m$, then $|N_{Py(m, r, s)}(w) \cap S| = 0$.
Given $w \in \{u_i\}_{i=1}^m$, if $w \in \{u_{i_0}, u_{i_0+1}\}$, then $|N_{Py(m, r, s)}(w) \cap S| = 2$, and otherwise $|N_{Py(m, r, s)}{}(w) \cap S| = 3$. Therefore, $S$ is a fort.
\end{proof}


\begin{claim}\label{peony:claim4}
Let $T = \{(i, j) | i \in \{1, 2, \dots, m\}, j \in \{1, 2, \dots, r\} \}$.
For each $p = (i, j) \in T$, let $v_{i,j,k_p} \in S_{i,j}$.
If $B = \{c\} \cup \{v_{i,j,k_p}\}_{i = 1,}^m{}_{j = 1}^r$, then $V(Py(m, r, s)) \backslash B$ is a fort of $Py(m, r, s)$.
\end{claim}

\begin{proof}
Let $w \in B$.
    \begin{case}$w = c$
    \end{case}
    It follows that $N_{Py(m, r, s)}(w) \cap (V(Py(m, r, s)) \backslash B) = \{u_i\}_{i = 1}^m$ and $|\{u_i\}_{i = 1}^m| = m > 1$.
    
    \begin{case} $w \in \{v_{i,j,k_p}\}_{1 = 1,}^m{}_{j = 1}^r$
    \end{case}
    First fix $w = v_{i_0, j_0, k_0}$ where $i_0 \in [m]$, $j_0 \in [r]$, and $k_0$ denotes $k_p$ with $p = (i_0, j_0)$.
      \begin{itemize}
       \item If $s = 1$, then
            $|N_{Py(m, r, s)}(w) \cap (V(Py(m, r, s))\backslash B)| = |N_{Py(m, r, s)}(w) \cap \{u_i\}_{i=1}^m| = 2$.
       \item If $s = 2$, then
            $|N_{Py(m, r, s)}(w) \cap \{u_i\}_{i=1}^m| = 1$ and $|N_{Py(m, r, s)}(w) \cap \{v_{i_0,j_0,k}\}_{k \in [s]\backslash \{k_0\}}| = 1$.
            Since $\{u_i\}_{i=1}^m \cap \{v_{i_0,j_0,k}\}_{k \in [s]\backslash \{k_0\}} = \emptyset$, it follows that $|N_{Py(m, r, s)}(w) \cap (V(Py(m, r, s)) \backslash B)| = 2$.
       \item Now suppose $s \geq 3$.  If $k_p \in \{1, s\}$, then $|N_{Py(m, r, s)}(w) \cap \{u_i\}_{1 = 1}^m| = 1$ and $|N_{Py(m, r, s)}(w) \cap \{v_{i_0,j_0,k}\}_{k \in [s]\backslash \{k_0\}}| = 1$. 
            So $|N_{Py(m, r, s)}(w) \cap (V(Py(m, r, s)) \backslash B)| = 2$.
            Otherwise, $|N_{Py(m, r, s)}(w) \cap \{v_{i_0,j_0,k}\}_{k \in [s]\backslash \{k_0\}}| = 2$. 
      \end{itemize}     
  
Therefore, $V(Py(m, r, s)) \backslash B$ is a fort.
\end{proof}

From \Cref{thm:fort} we see that if $Py(m, r, s) \backslash B$ contains a fort, then $B$ cannot be a zero forcing set of $Py(m, r, s)$. Using this fact, we prove the following theorem.

\begin{theorem}\label{theorem:peony_lower}
Let $Py(m, r, s)$ be a peony graph. Then $|Z(Py(m, r, s))| \geq m(r - 1) + 3$.
\end{theorem}

For clarity of proof, we use the term type-n fort to refer to a fort stated in \textbf{Claim n} with n $\in \{1, 2, 3, 4\}$.

\begin{proof} 
To prove this result, we will construct a set of minimum cardinality $B \subseteq V(Py(m,r,s))$ which intersects all type-n forts for $n \in \{1,2,3\}$.   Noting that $\abs{B}=m(r-1)+2$, we will then identify a type-4 fort $F$ for which $B \cap F=\emptyset$, completing the proof.

First, we construct a set, $B_1$, which intersects all type-\ref{peony:claim1} forts. Since type-\ref{peony:claim1} forts are comprised of two layers from the same station, to intersect all of them, one vertex must be taken from each of $r-1$ layers at every station.  Since the collection of stations $\{\mathcal S_i\}_{i=1}^m$ is pairwise disjoint, and furthermore the collection of layers $\{S_{i,j}\}_{j=1}^r$ at the $i$-th station is also pairwise disjoint, this will require at least $m(r-1)$ vertices, specifically, $r-1$ vertices chosen from distinct layers at each of the $m$ stations. To construct an arbitrary set $B_1$ of this type, for each station $i$, choose $j_i \in [r]$ and for each $S_{i, j}$ with $j \neq j_i$, choose $k_{i,j} \in [s]$. Let $B_1 = \{v_{i, j, k_{i, j}}\}_{i=0,}^{m}{}_{j \neq j_i}$. Note that $B_1$ is of the minimum cardinality necessary to hit all of the type-\ref{peony:claim1} forts. 

Second, we construct a set, $B_2$, containing $B_1$, which intersects all type-\ref{peony:claim1} and type-\ref{peony:claim2} forts. Since type-\ref{peony:claim2} forts take a layer from every station, and for each station $i$, there is a layer $S_{i, j_i}$ for which $B_1 \cap S_{i, j_i}  =  \emptyset $, it follows that $\bigcup_{i=1}^m S_{i,j_i}$ is a type-\ref{peony:claim2} fort for which $B_1 \cap \bigcup_{i=1}^m S_{i,j_i} = \emptyset$. With this in mind, choose $i_2 \in [m]$ and $k_2 \in [s]$, and let $B_2 = B_1 \cup \{v_{{i_2}, j_{i_2}, k_2}\}$. Note, the choice of $B_1$ was arbitrary and we have shown that for each choice of $B_1$, a type-2 fort disjoint from $B_1$ can be identified. Thus, it follows that to hit all type-\ref{peony:claim1} and type-\ref{peony:claim2} forts, will require at least $m(r-1)+1$ vertices.

Next, we construct a set, $B_3$, containing $B_2$, which intersects all type-\ref{peony:claim1}, type-\ref{peony:claim2}, and type-\ref{peony:claim3} forts. Since the center vertex and a layer from each of $m-1$ distinct stations form a type-\ref{peony:claim3} fort, and for each station $i$ with $i \neq i_2$ there is a layer $S_{i, j_i}$ for which $B_2 \cap S_{i, j_i} = \emptyset$, it follows that $\bigcup_{i\neq i_2} S_{i,j_i} \cup \{c\}$ is a type-\ref{peony:claim3} fort for which $B_2 \cap\left( \bigcup_{i\neq i_2} S_{i,j_i} \cup \{c\}\right) = \emptyset$. Thus, the possible alternatives are either $B_3 = B_2 \cup \{c\}$ or $B_3 = B_2 \cup \{v_{i_3, j_{i_3}, k_3}\}$ for arbitrary chosen $i_3 \in [m]\backslash\{i_2\}$ and $k_3 \in [s]$. Note that the choice of $B_2$ was arbitrary and for any choice of $B_2$, a type-\ref{peony:claim3} fort disjoint from $B_2$ can be identified. 
 Either way, to not miss any type-\ref{peony:claim1}, type-\ref{peony:claim2}, or type-\ref{peony:claim3} forts, $B_3$ must contain at least $m(r-1)+2$ vertices.

 Finally, since $B_3$ has been constructed by taking at most one vertex from each layer (as well as possibly the center vertex), by Claim \ref{peony:claim4}, $V(Py(m,r,s))\setminus B_3$ contains a type-\ref{peony:claim4} fort $F$.  Thus, any set intersecting all type-$n$ forts for $n \in \{1,2,3,4\}$ must contain at minimum $m(r-1)+3$ vertices.
\end{proof}

\begin{theorem}
$Py(m, r, s)$ be a peony graph. Then $|\Z(Py(m, r, s))| = m(r-1) + 3$.
\end{theorem}

\begin{proof}
Let $B = \{v_{i,j,1}\}_{i = 2,}^{m}{}_{j = 2}^{r}\cup\{v_{m,j,s}\}_{j=1}^r\cup\{c\}\cup\{u_1\}$ be the set of vertices initially colored blue. Note that $\abs{\{v_{i,j,1}\}_{i = 2,}^{m}{}_{j = 2}^{r}\cup\{v_{m,j,s}\}_{j=1}^r} = (m-1)(r-1)+r=m(r-1)+1$, so $\abs{B} = m(r-1) + 3$.

We now construct a relaxed chronology of forces $\mathcal F$ of $B$ on $Py(m,r,s)$.  On the first time-step, the vertex $v_{1,1,1}$ is the unique white neighbor of $u_1$, so we can let $u_1$ force $v_{1,1,1}$, and after the first time-step $v_{1,1,1}$ will be blue.  On time-step $p$ with $2 \leq p \leq s$, because $\{v_{1, j, p-1}\}_{j = 1}^{r}$ is a collection of blue vertices and for each $j \in [r]$, $v_{1, j, p}$ is the unique white neighbor of $v_{1, j, p-1}$, we can perform the forces $\{v_{1, j, p-1} \rightarrow v_{1, j, p}\}_{j = 1}^{r}$. Thus, after time-step $s$, all vertices in the first station are blue.  Similar forces occur in the $m$-th station.  On time-step $p$ with $2 \leq p \leq s$, because $\{v_{m,j,s-p+2}\}_{j=1}^r$ is a collection of blue vertices and for each $j \in [r]$, $v_{m, j, s-p+1}$ is the unique white neighbor of $v_{m, j, s-p+2}$, we can perform the forces $\{v_{m, j, s-p+2} \rightarrow v_{m, j, s-p+1}\}_{j=1}^r$. 
On time-step $s+1$, the vertex $u_m$ is the unique white neighbor of $v_{m, 1, 1}$, so we can let $v_{m, 1, 1}$ force $u_m$. 
Note that after time-step $s+1$, both the first and $m$-th stations are blue. 

To color the entire second station, $u_2$ first must be made blue. On time-step $s+2$, since the vertex $u_2$ is the unique white neighbor of $v_{1, 1, s}$, $v_{1, 1, s}$ forces $u_2$. The remaining process of coloring the second station blue is identical to that which occurred in the first station.  Now, repeat the aforementioned process $m-3$ times, thus forcing stations $3$ through $m-1$ to become blue, and at this point no white vertices remain in $Py(m, r, s)$.

Since $B$ is a zero forcing set of $Py(m, r , s)$ and $\abs{B} = m(r-1) + 3 $, it follows that $|\Z(Py(m, r, s))| \leq m(r - 1) + 3$.  Thus, by \Cref{theorem:peony_lower}, we can conclude that $|\Z(Py(m, r, s))| = m(r - 1) + 3$.

\end{proof}

\section{The zero forcing numbers of web graphs}

\quad We now introduce a class of graphs which we call web graphs.  It is worth noting that the definition we provide here is more generalized than  another definition which is commonly used, and sometimes the parameter $r$ used to define our more generalized version is fixed such that $r=2$.

\begin{definition}
We define the web graph $Wb(m,r)$ to be the graph with vertex set $V(Wb(m,r))=\{v_{i, j}\}_{i=1,}^m{}_{j=1}^r \cup\{p_i\}_{i=1}^m$ such that for $u,w \in V(Wb(m,r))$ distinct, $uw \in E(Wb(m,r))$ if and only if one of the following is true:
\begin{itemize}
\item $\{u,w\}=\{p_i,v_{i,1}\}$
\item $i_1=i_2$ and $\abs {j_1-j_2}=1$
\item $j_1=j_2$, and either $\{i_1,i_2\}=\{1, m\}$ or $\abs{i_1-i_2}=1$.
\end{itemize}
\end{definition}

\begin{figure}[htbp]
    \centering
    \begin{minipage}[b]{0.8\textwidth} 
        \centering

        \begin{tikzpicture}[scale=1] 
    \draw(0,0) circle (3 and 1.5);
    \draw(0,1) circle (3 and 1.5);
    \draw(0,2) circle (3 and 1.5);
    
    \fill (-3, 0) circle (2pt);
    \node[anchor=south east] at (-3, 0) {$v_{1,3}$}; 
    \fill (-3, 1) circle (2pt);
    \node[anchor=south east] at (-3, 1) {$v_{1, 2}$}; 
    \fill (-3, 2) circle (2pt);
    \node[anchor=south east] at (-3, 2) {$v_{1, 1}$}; 
    \fill (-3, 3.5) circle (2pt);
    \node[anchor=south east] at (-3, 3.5) {$p_1$}; 
    \draw (-3, 0) -- (-3, 3.5);

    \fill (3, 0) circle (2pt);
    \node[anchor=south west] at (3, 0) {$v_{4,3}$}; 
    \fill (3, 1) circle (2pt);
    \node[anchor=south west] at (3, 1) {$v_{4,2}$}; 
    \fill (3, 2) circle (2pt);
    \node[anchor=south west] at (3, 2) {$v_{4,1}$}; 
    \fill (3, 3.5) circle (2pt);
    \node[anchor=south west] at (3, 3.5) {$p_4$}; 
    \draw (3, 0) -- (3, 3.5);

    \fill (0, -1.5) circle (2pt);
    \node[anchor=south west] at (0, -1.5) {$v_{5, 3}$}; 
    \fill (0, -0.5) circle (2pt);
    \node[anchor=south west] at (0, -0.5) {$v_{5, 2}$}; 
    \fill (0,  0.5) circle (2pt);
    \node[anchor=south west] at (0,  0.5) {$v_{5, 1}$}; 
    \fill (0,    2) circle (2pt);
    \node[anchor=south west] at (0,    2) {$p_5$}; 
    \draw (0, -1.5) -- (0, 2);

    \fill (1.5, 1.3) circle (2pt);
    \node[anchor=south west] at (1.5, 1.3) {$v_{3, 3}$}; 
    \fill (1.5, 2.3) circle (2pt);
    \node[anchor=south west] at (1.5, 2.3) {$v_{3, 2}$}; 
    \fill (1.5, 3.3) circle (2pt);
    \node[anchor=south west] at (1.5, 3.3) {$v_{3, 1}$}; 
    \fill (1.5, 4.7) circle (2pt);
    \node[anchor=south west] at (1.5, 4.7) {$p_3$}; 
    \draw (1.5, 1.3) -- (1.5, 4.7);

    \fill (-1.5, 1.3) circle (2pt);
    \node[anchor=south west] at (-1.5, 1.3) {$v_{2, 3}$}; 
    \fill (-1.5, 2.3) circle (2pt);
    \node[anchor=south west] at (-1.5, 2.3) {$v_{2, 2}$}; 
    \fill (-1.5, 3.3) circle (2pt);
    \node[anchor=south west] at (-1.5, 3.3) {$v_{2, 1}$}; 
    \fill (-1.5, 4.7) circle (2pt);
    \node[anchor=south west] at (-1.5, 4.7) {$p_2$}; 
    \draw (-1.5, 1.3) -- (-1.5, 4.7);
\end{tikzpicture}
        
        \caption{The web graph $Wb(5,3)$}
        \label{fig:web}
    \end{minipage}
\end{figure}
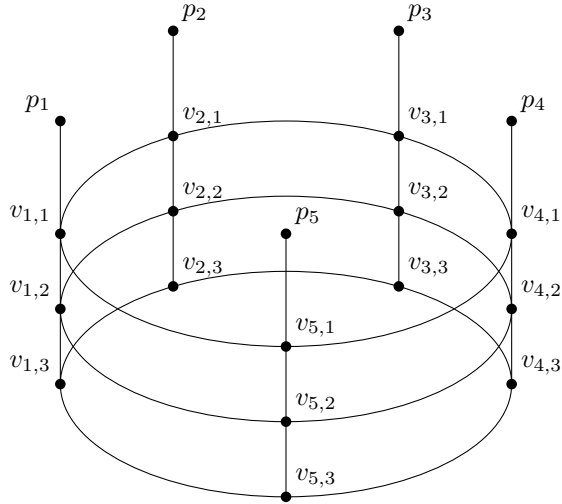

Alternatively this can be viewed as a result of a small modification to the graph $V(C_m\Box P_r)$, where the {\em Cartesian product} of two graphs $G$ and $H$ denoted $G \Box H$ is defined to be the graph with vertex set $V(G) \times V(H)$ such that two vertices $(g_1,h_1),(g_2,h_2) \in V(G \Box H)$ are adjacent in $G \Box H$ if either $g_1=g_2$ and $h_1h_2 \in E(H)$ or $h_1=h_2$ and $g_1g_2 \in E(G)$.  Specifically, $Wb(m,r)$ is the result of adding the set $\{p_i\}_{i=1}^m$ of pendant vertices to the graph $C_m\Box P_r$, with vertex set $V(C_m \Box P_r)=\{v_{i,j}\}_{i=1,}^m{}_{j=1}^r$, such that for each $i$, $p_i$ is adjacent to $v_{i,1}$.  To determine the zero forcing numbers of web graphs, we first identify a pair of commonly cited results which will be key to our argument.

\begin{lemma}\cite{aim}\label{prism}
$\Z(C_m \Box P_r)=\min\{m,2r\}$.
\end{lemma}

\begin{lemma}\cite{param}\label{pvsz}
$\p(G) \leq \Z(G)$.  Furthermore, if $\mathcal C$ is a chain set for some relaxed chronology of forces $\mathcal F$ for some zero forcing set $B$ of $G$, then $\mathcal C$ is a path cover of $G$.
\end{lemma}

Recent work in the area has been focused on identifying and studying the interaction between local properties of a graph (those confined to a particular subgraph) and global properties of a graph and the affect this interaction has on zero forcing.  One such concept introduced in \cite{PIP} involves restricting relaxed chronologies of forces to those forces in a given subgraph.  Specifically, given a graph $G$ and $H$ a vertex-induced subgraph of $G$, if $\mathcal F=\{F^{(k)}\}_{k=1}^K$ is a relaxed chronology of forces of some zero forcing set $B$ of $G$, then the {\em restriction} of $\mathcal F$ to $H$ denoted $\mathcal F|_H=\{F|_H^{(k)}\}_{k=1}^K$ is the relaxed chronology of forces such that at each time-step $k$ we have $F_H^{(k)}=\{u \rightarrow v: u,v \in V(H) \text{ and } u\rightarrow v \in F^{(k)}\}$.  Letting $\mathcal C$ be the chain set induced by $\mathcal F$, the {\em forcing subpaths} of $\mathcal C$ in $H$ are the subgraphs of $H$ induced by the vertices in each member of $\mathcal C$.  A vertex $u\in V(H)$ is an \emph{initial vertex} of a forcing subpath of $\mathcal C$ in $H$ if either $u \in B$ or $v \rightarrow u$ is a force which occurs during $\mathcal F$ but $v \in V(G) \setminus V(H)$.  The following result concerning initial vertices helps us to prove a lower bound on the zero forcing numbers of web graphs.

\begin{lemma}\cite{PIP}\label{restriction}
Let $G$ be a graph, $H$ be a vertex-induced subgraph of $G$, $B$ be a standard zero forcing set of $G$, and $\mathcal{F}$ be a relaxed chronology of standard forces of $B$ on $G$ with chain set $\mathcal{C}$. Let $B'$ be the set of initial vertices in the forcing subpaths of $\mathcal{C}$ in $H$. Then $B'$ is a standard zero forcing set of $H$, and $\mathcal{F}|_H$ defines a relaxed chronology of standard forces for $B'$ in $H$.
\end{lemma}

\begin{lemma}\label{weblower}
$\Z(Wb(m,r)) \geq \max\left\{\left\lceil\frac{m}{2}\right\rceil, \min\{m,2r\}\right\}$.
\end{lemma}

\begin{proof}
The graph $Wb(m,r)$ contains $m$ pendant vertices.  If $\mathcal Q$ is a path cover of $Wb(m,r)$, then since $\mathcal Q$ forms a partition of $V(Wb(m,r))$ each of these $m$ pendant vertices must be contained in some $Q \in \mathcal Q$.  Furthermore, for each $Q \in \mathcal Q$, $Q$ contains at most 2 pendant vertices as the vertices which are not endpoints of $Q$ each have two distinct neighbors in $Q$. So by \ref{pvsz}, $\Z(Wb(m,r)) \geq P(Wb(m,r)) \geq \left\lceil\frac{m}{2}\right\rceil$.

We will now show, by way of contradiction, that $\Z(Wb(m,r)) \geq \min\{m,2r\}$.  Suppose $B$ is a minimum zero forcing set of $Wb(m,r)$ with $\abs{B}<\min\{m,2r\}$.  Let $\mathcal F$ be a relaxed chronology of forces of $B$ on $Wb(m,r)$ and let $\mathcal C$ be the chain set induced by $\mathcal F$.  Now let $H$ be the subgraph of $Wb(m,r)$ with $H \cong C_m\Box P_r$, let $\mathcal F|_H$ be the restriction of $\mathcal F$ to $H$, and let $\mathcal C_H$ be the collection of forcing subpaths of $\mathcal C$ contained in $H$.  By \ref{restriction}, if $B_H$ is the set of initial vertices of $\mathcal C_H$ in $H$, then $B_H$ is a zero forcing set of $H$.  Since pendant vertices can only be the endpoints of paths, and thus the endpoints of forcing chains, each forcing chain $C \in \mathcal C$ contains at most one member of $\mathcal C_H$ as a forcing subpath.  So it follows that $B_H$ is a zero forcing set of $H\cong C_m \Box P_r$ with $\abs{B_H}=\abs{\mathcal C_H} \leq \abs{\mathcal C}=\abs{B}<\min\{m,2r\}$, a contradiction of \Cref{prism}.  Thus $Wb(m,r) \geq \min\{m,2r\}$.
\end{proof}

\begin{lemma}\label{ubound1}
If $m \leq 2r$, then $\Z(Wb(m,r)) \leq m$. 
\end{lemma}
\begin{proof}
Let $B = \{p_i\}_{i=1}^{m}$.  We will construct a relaxed chronology of forces $\mathcal F$ for $B$ on $Wb(m,r)$.
On time-step 1, $F^{(1)} = \{p_i \to v_{i,1}\}_{i=1}^{m}$. 
On time-step $k$ with $2 \leq k \leq r$, $F^{(k)} = \{v_{i, k-1} \to v_{i,k}\}_{i=1}^{m}$.  So after time-step $r$, we have that $E_{\mathcal F}^{[r]} = V(Wb(m,r))$.  Thus, $\Z(Wb(m,r)) \leq \abs{B}=m$
\end{proof}


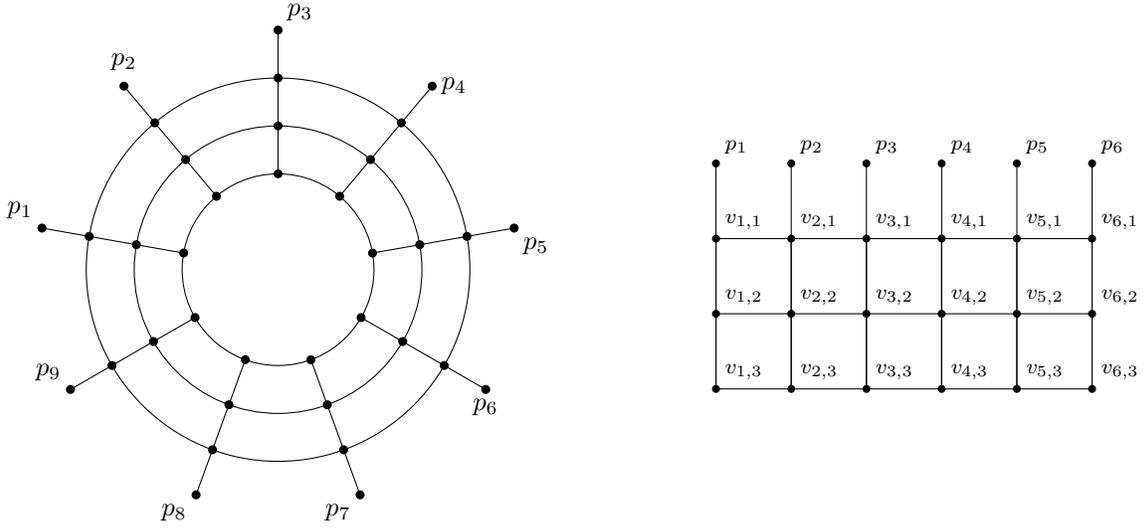
\begin{figure}[htbp]
    \centering
        \begin{minipage}[c]{0.5\textwidth} 
        \centering
        \begin{tikzpicture}[scale=0.85] 

            \draw(0,0) circle (3 and 3);
            \draw(0,0) circle (2.25 and 2.25);
            \draw(0,0) circle (1.5 and 1.5);

            \fill (-1.4772, 0.26047) circle (2pt);
            \fill (-2.216, 0.3907) circle (2pt);
            \fill (-2.9544, 0.5209) circle (2pt);
            \fill (-3.693, 0.6512) circle (2pt);
            \node[anchor=south east] at (-3.693, 0.6512) {$p_{1}$}; 
            \draw (-1.4772, 0.26047) -- (-3.693, 0.6512);

            \fill (-0.9642, 1.149) circle (2pt);
            \fill (-1.4463, 1.7236) circle (2pt);
            \fill (-1.9284, 2.298) circle (2pt);
            \fill (-2.4105, 2.8727) circle (2pt);
            \node[anchor=south] at (-2.4105, 2.8727+0.1) {$p_2$}; 
            \draw (-0.9642, 1.149) -- (-2.4105, 2.8727);

            \fill (0, 1.5) circle (2pt);
            \fill (0, 2.25) circle (2pt);
            \fill (0, 3) circle (2pt);
            \fill (0, 3.75) circle (2pt);
            \node[anchor=south west] at (0, 3.75) {$p_3$}; 
            \draw (0, 1.5) -- (0, 3.75);

            \fill (0.9642, 1.149) circle (2pt);
            \fill (1.4463, 1.7236) circle (2pt);
            \fill (1.9284, 2.298) circle (2pt);
            \fill (2.4105, 2.8727) circle (2pt);
            \node[anchor=west] at (2.4105, 2.8727) {$p_4$}; 
            \draw (0.9642, 1.149) -- (2.4105, 2.8727);

            \fill (1.4772, 0.26047) circle (2pt);
            \fill (2.216, 0.3907) circle (2pt);
            \fill (2.9544, 0.5209) circle (2pt);
            \fill (3.693, 0.6512) circle (2pt);
            \node[anchor=north west] at (3.693, 0.6512) {$p_{5}$}; 
            \draw (1.4772, 0.26047) -- (3.693, 0.6512);

            \fill (1.299, -0.75) circle (2pt);
            \fill (1.9486, -1.125) circle (2pt);
            \fill (2.598, -1.5) circle (2pt);
            \fill (3.2476, -1.875) circle (2pt);
            \node[anchor=north] at  (3.2476, -1.875) {$p_6$}; 
            \draw (1.299, -0.75) -- (3.2476, -1.875);

            \fill (0.513, -1.4095) circle (2pt);
            \fill (0.7695, -2.1143) circle (2pt);
            \fill (1.026, -2.819) circle (2pt);
            \fill (1.2826, -3.524) circle (2pt);
            \node[anchor=north east] at (1.2826, -3.524) {$p_7$}; 
            \draw (0.513, -1.4095) -- (1.2826, -3.524);

            \fill (-0.513, -1.4095) circle (2pt);
            \fill (-0.7695, -2.1143) circle (2pt);
            \fill (-1.026, -2.819) circle (2pt);
            \fill (-1.2826, -3.524) circle (2pt);
            \node[anchor=north east] at (-1.2826, -3.524) {$p_8$}; 
            \draw (-0.513, -1.4095) -- (-1.2826, -3.524);

            \fill (-1.299, -0.75) circle (2pt);
            \fill (-1.9486, -1.125) circle (2pt);
            \fill (-2.598, -1.5) circle (2pt);
            \fill (-3.2476, -1.875) circle (2pt);
            \node[anchor=south east] at  (-3.2476, -1.875) {$p_9$}; 
            \draw (-1.299, -0.75) -- (-3.2476, -1.875);

        \end{tikzpicture}
        \label{fig:upperbound_3.6}
    \end{minipage}
    \hfill
    \begin{minipage}[c]{0.45\textwidth}
        \centering
        \begin{tikzpicture}[scale=1] 
            \draw[black] (1,1) grid (6,3);
            \foreach \x in {1,...,6} {
                \foreach \y [evaluate=\y as \yy using {int(4 - \y)}] in {1,...,3}{
                    \fill (\x,\y) circle (1.5pt); 
                    \node[anchor=south west, font=\footnotesize] at (\x,\y) {$v_{\x,\yy}$}; 
                }
            }
            \foreach \p in {1,...,6}{
                \draw (\p, 4) -- (\p, 1);
                \fill (\p, 4) circle (1.5pt);
                \node[anchor=south west, font=\footnotesize] at (\p,4) {$p_{\p}$}; 
            }
            
        \end{tikzpicture}
        \label{fig:upperbound_3.7}
    \end{minipage}
    \caption{$Wb(9,3)$ and an example of the subgraph $H$ of $Wb(9,3)$ discussed in \Cref{ubound2}}
    \label{fig:upperbound_3.7}
\end{figure}

\begin{figure}[htbp]
    \centering
    \begin{minipage}[b]{0.45\textwidth}
        \centering

        \begin{tikzpicture}[scale=1.2] 
            \draw[black] (1,1) grid (6,3);
            \foreach \x in {1,...,6} {
                \foreach \y [evaluate=\y as \yy using {int(4 - \y)}] in {1,...,3}{
                    \draw[line width=0.4mm] (\x,\y) circle (2.5pt); 
                }
            }

            \foreach \p in {1,...,6}{
                \draw[->, >=latex, line width=1.5pt] (\p, 4) -- (\p, 3);
                \fill (\p, 4) circle (2.5pt);
                \node[anchor=south west, font=\footnotesize] at (\p,4) {$p_{\p}$}; 
                
                \ifnum\p=2
                    \draw[->, >=latex, line width=1.5pt] (\p, 3) -- (\p, 2);
                \else
                    \ifnum\p=3
                        \draw[->, >=latex, line width=1.5pt] (\p, 3) -- (\p, 2);
                        \draw[->, >=latex, line width=1.5pt] (\p, 2) -- (\p, 1);
                    \else
                        \ifnum\p=4
                            \draw[->, >=latex, line width=1.5pt] (\p, 3) -- (\p, 2);
                            \draw[->, >=latex, line width=1.5pt] (\p, 2) -- (\p, 1);
                        \else
                            \ifnum\p=5
                                \draw[->, >=latex, line width=1.5pt] (\p, 3) -- (\p, 2);
                             \fi
                        \fi
                    \fi
                \fi
            }

        \end{tikzpicture}
       
        \label{fig:graph1}
    \end{minipage}
    \hfill
    \begin{minipage}[b]{0.45\textwidth}
        \centering

        \begin{tikzpicture}[scale=1.2] 
            \draw[black] (1,1) grid (6,3);
            \foreach \x in {1,...,6} {
                \foreach \y [evaluate=\y as \yy using {int(4 - \y)}] in {1,...,3}{
                    \draw[line width=0.4mm] (\x,\y) circle (2.2pt); 
                }
            }
            \foreach \p in {1,...,6}{
                \draw (\p, 4) -- (\p, 1);
                \fill (\p, 4) circle (2.5pt);
                \node[anchor=south west, font=\footnotesize] at (\p,4) {$p_{\p}$};
            }

            \foreach \x in {1,...,6}{
                \fill (\x, 3) circle (1.8pt);
            }
            \foreach \x in {2,...,5}{
                \fill (\x, 2) circle (1.8pt);
            }
            \foreach \x in {3,...,4}{
                \fill (\x, 1) circle (1.8pt);
            }
            
             \foreach \p in {1,...,2}{
                
                \ifnum\p=1
                    \draw[->, >=latex, line width=1.5pt] (2, \p) -- (1, \p);
                    \draw[->, >=latex, line width=1.5pt] (3, \p) -- (2, \p);
                    \draw[->, >=latex, line width=1.5pt] (4, \p) -- (5, \p);
                    \draw[->, >=latex, line width=1.5pt] (5, \p) -- (6, \p);
                \else
                    \ifnum\p=2
                        \draw[->, >=latex, line width=1.5pt] (2, \p) -- (1, \p);
                        \draw[->, >=latex, line width=1.5pt] (5, \p) -- (6, \p);
                    \fi
                \fi
            }

        \end{tikzpicture}
       
        \label{fig:graph2}
    \end{minipage}

    \caption{The first $2r-1=5$ time-steps of $\mathcal F$ in $Wb(9,3)$}
    \label{fig:upperbound_3.7}
\end{figure}
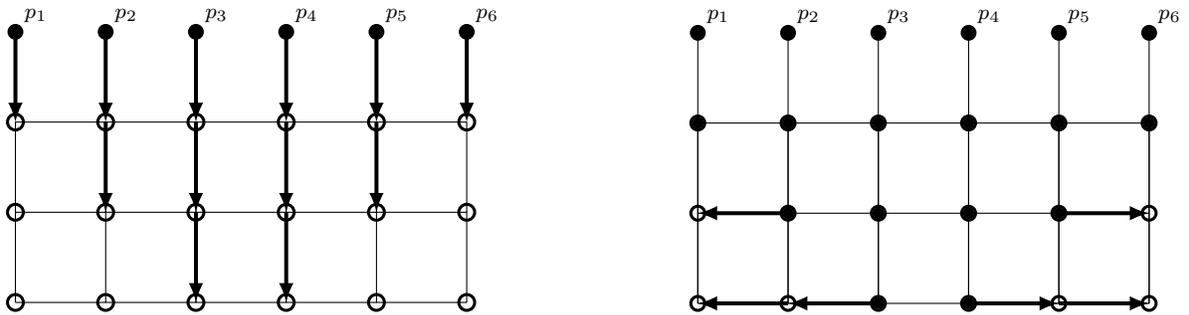

\begin{lemma}\label{ubound2}
If $\lceil \frac{m}{2}\rceil < 2r < m$, then $\Z(Wb(m,r)) \leq 2r$. 
\end{lemma}
\begin{proof}
Let $B = \{p_i\}_{i = 1}^{2r}$.  We will construct a relaxed chronology of forces $\mathcal F$ for $B$ on $Wb(m,r)$.  On time-step 1, $F^{(1)} = \{p_i \to v_{i,1}\}_{i=1}^{2r}$.  Moving forward, for $k$ with $2 \leq k \leq 2r$, at each successive time-step $k$, fewer vertices in the set $\{v_{i,k-1}\}_{i=1}^{2r}$ will be able to force as an increasing number of vertices in these sets will have more than one white neighbor.  However, forcing will continue.  Specifically, for $k$ with $2 \leq k \leq r$, during time-step $k$ the forces in the set $\{v_{i, k-1} \to v_{i, k}\}_{i=k}^{2r-k+1}$ can be performed because $v_{i, k}$ is the only white neighbor of $v_{i, k-1}$.  On the other hand, once these forces are completed, an increasing number of vertices will be able to force during each successive time-step $k$, with $r+1 \leq k \leq 2r-1$.  Specifically, for $k$ with $r+1 \leq k \leq 2r-1$, during time-step $k$ the forces in the set $\{v_{2r-k+1, j} \to v_{2r-k, j},v_{k, j} \to v_{k+1, j}\}_{j=2r-k+1}^r$ can be performed because $v_{2r-k,j}$ is the only white neighbor of $v_{2r-k+1,j}$ and $v_{k+1,j}$ is the only white neighbor of $v_{k,j}$.  

Now identify that $E_{\mathcal F}^{[2r-1]}=\{v_{i,j}\}_{i=1,}^{2r}{}_{j=1}^{r} \cup \{p_i\}_{i=1}^{2r}$, and let $H$ be the subgraph of $Wb(m,r)$ induced by the vertices which are blue after time-step $2r-1$, specifically, $H=Wb(m,r)[E_{\mathcal F}^{[2r-1]}]$.  Since every vertex in $H$ is blue after time-step $2r-1$, $B$ is a zero forcing set of $H$.  Furthermore, $\mathcal F|_H$ is a relaxed chronology of forces of $B$ on $H$, so by \Cref{term}, $\Term(\mathcal F|_H)$ is also a zero forcing set of $H$.  

Now let $H'=Wb(m,r)[\{v_{i,j}\}_{i=m-2r+2,}^m{}_{j=1}^r \cup \{v_{1,j}\}_{j=1}^r \cup \{p_i\}_{i=m-2r+2}^m \cup \{p_1\}]$ and note that $H \cong H'$ under some isomorphism $\sigma$.  Due to this, $\sigma(\Term(\mathcal F|_H))$ is a zero forcing set of $H'$.  Furthermore, since $\left\lceil\frac{m}{2}\right\rceil<2r$, it follows that $m-2r+2 \leq 2r$ and the vertex sets of $H$ and $H'$ overlap in such a way that $\sigma(\Term(\mathcal F|_H)) \subseteq E_{\mathcal F}^{[2r-1]}$.  It then follows that $E_{\mathcal F}^{[2r-1]}$ contains a zero forcing set of $H'$.  Since $Wb(m,r)$ contains no blue vertices outside of $H'$ after time-step $2r-1$, the forcing process can continue and $B$ must be a zero forcing set of $Wb(m,r)$.  Thus, $\Z((Wb(m,r))\leq \abs{B}=2r$.  
\end{proof}


\begin{figure}[htbp]
    \centering
    \begin{minipage}[b]{0.8\textwidth} 
        \centering
        \begin{tikzpicture}[scale=1] 

            \draw(0,0) circle (3 and 3);
            \draw(0,0) circle (2.25 and 2.25);
            \draw(0,0) circle (1.5 and 1.5);

            \draw (-1.5, 0) circle (2pt);
            \draw (-2.25, 0) circle (2pt);
            \draw (-3, 0) circle (2pt);
            \fill (-3.75, 0) circle (2.4pt);
            \node[anchor=east] at (-3.75, 0) {$p_1$}; 
            \draw (-1.5, 0) -- (-3.75, 0);

            \draw (-1.3858, 0.574) circle (2pt);
            \draw (-2.0787, 0.861) circle (2pt);
            \draw (-2.7716, 1.148) circle (2pt);
            \fill (-3.4645, 1.435) circle (2.4pt);
            \node[anchor=south east] at (-3.4645, 1.435-0.1) {$p_{2}$}; 
            \draw (-1.3858, 0.574) -- (-3.4645, 1.435);

            \draw (-1.06, 1.06) circle (2pt);
            \draw (-1.59, 1.59) circle (2pt);
            \draw (-2.12, 2.12) circle (2pt);
            \fill (-2.65, 2.65) circle (2.4pt);
            \node[anchor=south east] at (-2.65, 2.65) {$p_3$}; 
            \draw (-1.06, 1.06) -- (-2.65, 2.65);

            \draw (-0.574, 1.3858) circle (2pt);
            \draw (-0.861, 2.0787) circle (2pt);
            \draw (-1.148, 2.7716) circle (2pt);
            \fill (-1.435, 3.4645) circle (2.4pt);
            \node[anchor=south] at (-1.435, 3.4645) {$p_4$}; 
            \draw (-0.574, 1.3858) -- (-1.435, 3.4645);

            \draw (0, 1.5) circle (2pt);
            \draw (0, 2.25) circle (2pt);
            \draw (0, 3) circle (2pt);
            \fill (0, 3.75) circle (2.4pt);
            \node[anchor=south] at (0, 3.75) {$p_5$}; 
            \draw (0, 1.5) -- (0, 3.75);
            
            \draw (0.574, 1.3858) circle (2pt);
            \draw (0.861, 2.0787) circle (2pt);
            \draw (1.148, 2.7716) circle (2pt);
            \fill (1.435, 3.4645) circle (2.4pt);
            \node[anchor=south west] at (1.435, 3.4645) {$p_6$}; 
            \draw (0.574, 1.3858) -- (1.435, 3.4645);

            \draw (1.06, 1.06) circle (2pt);
            \draw (1.59, 1.59) circle (2pt);
            \draw (2.12, 2.12) circle (2pt);
            \draw (2.65, 2.65) circle (2pt);
            \node[anchor=west] at (2.65, 2.65) {$p_7$}; 
            \draw (1.06, 1.06) -- (2.65, 2.65);

            \draw (1.3858, 0.574) circle (2pt);
            \draw (2.0787, 0.861) circle (2pt);
            \draw (2.7716, 1.148) circle (2pt);
            \fill (3.4645, 1.435) circle (2.4pt);
            \node[anchor=west] at (3.4645, 1.435) {$p_8$}; 
            \draw (1.3858, 0.574) -- (3.4645, 1.435);

            \draw (1.5, 0) circle (2pt);
            \draw (2.25, 0) circle (2pt);
            \draw (3, 0) circle (2pt);
            \draw (3.75, 0) circle (2pt);
            \node[anchor=west] at (3.75, 0) {$p_9$}; 
            \draw (1.5, 0) -- (3.75, 0);

            \draw (1.3858, -0.574) circle (2pt);
            \draw (2.0787, -0.861) circle (2pt);
            \draw (2.7716, -1.148) circle (2pt);
            \fill (3.4645, -1.435) circle (2.4pt);
            \node[anchor=north west] at (3.4645, -1.435) {$p_{10}$}; 
            \draw (1.3858, -0.574) -- (3.4645, -1.435);

            \draw (1.06, -1.06) circle (2pt);
            \draw (1.59, -1.59) circle (2pt);
            \draw (2.12, -2.12) circle (2pt);
            \draw (2.65, -2.65) circle (2pt);
            \node[anchor=north west] at (2.65, -2.65) {$p_{11}$}; 
            \draw (1.06, -1.06) -- (2.65, -2.65);

            \draw (0.574, -1.3858) circle (2pt);
            \draw (0.861, -2.0787) circle (2pt);
            \draw (1.148, -2.7716) circle (2pt);
            \draw (1.435, -3.4645) circle (2pt);
            \node[anchor=north] at (1.435, -3.4645) {$p_{12}$}; 
            \draw (0.574, -1.3858) -- (1.435, -3.4645);

            \draw (0, -1.5) circle (2pt);
            \draw (0, -2.25) circle (2pt);
            \draw (0, -3) circle (2pt);
            \draw (0, -3.75) circle (2pt);
            \node[anchor=north] at (0, -3.75) {$p_{13}$}; 
            \draw (0, -1.5) -- (0, -3.75);

            \draw (-0.574, -1.3858) circle (2pt);
            \draw (-0.861, -2.0787) circle (2pt);
            \draw (-1.148, -2.7716) circle (2pt);
            \draw (-1.435, -3.4645) circle (2pt);
            \node[anchor=north] at (-1.435, -3.4645) {$p_{14}$}; 
            \draw (-0.574, -1.3858) -- (-1.435, -3.4645);

            \draw (-1.06, -1.06) circle (2pt);
            \draw (-1.59, -1.59) circle (2pt);
            \draw (-2.12, -2.12) circle (2pt);
            \draw (-2.65, -2.65) circle (2pt);
            \node[anchor=north east] at (-2.65, -2.65) {$p_{15}$}; 
            \draw (-1.06, -1.06) -- (-2.65, -2.65);

            \draw (-1.3858, -0.574) circle (2pt);
            \draw (-2.0787, -0.861) circle (2pt);
            \draw (-2.7716, -1.148) circle (2pt);
            \draw (-3.4645, -1.435) circle (2pt);
            \node[anchor=north east] at (-3.4645, -1.435) {$p_{16}$}; 
            \draw (-1.3858, -0.574) -- (-3.4645, -1.435);

        \end{tikzpicture}
        \caption{The zero forcing set $B$ of $Wb(16,3)$}
        \label{fig:upperbound_3.8_3}
    \end{minipage}
\end{figure}
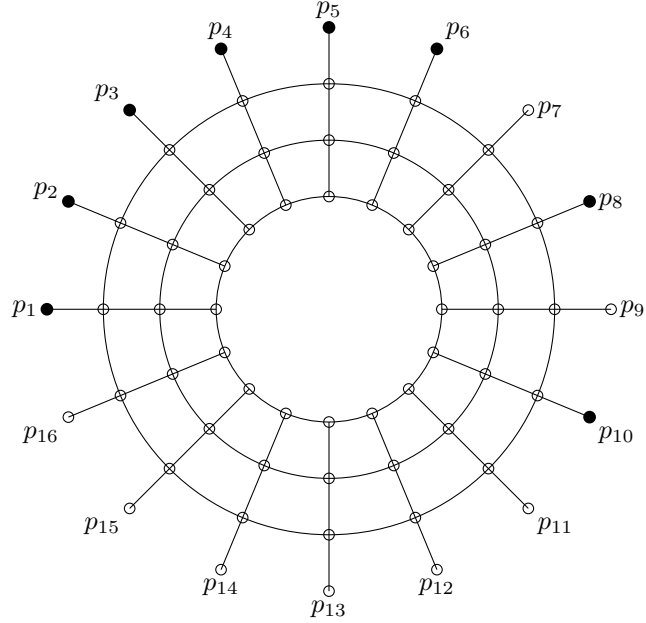

\begin{lemma}\label{ubound3}
If $2r \leq \lceil \frac{m}{2}\rceil$, then $\Z(Wb(m,r)) \le \lceil \frac{m}{2}\rceil$.
\end{lemma}
\begin{proof}
Construct the zero forcing set $B$ using the following pendant vertices.  First take $\{p_i\}_{i=1}^{2r}$ and then also include every second pendant vertex starting at $p_{2r+2}$ up until vertex $p_{m-2r}$, making sure to include vertex $p_{m-2r}$ even in the case where $m$ is odd.  Specifically, in the case that $m$ is even $B=\{p_i\}_{i=1}^{2r}\cup\{p_{2r+2i}\}_{i=1}^{(m-4r)/2}$, and in the case that $m$ is odd $B=\{p_i\}_{i=1}^{2r}\cup\{p_{2r+2i}\}_{i=1}^{(m-1-4r)/2}\cup\{p_{m-2r}\}$; and in either case $\abs{B}=\left\lceil\frac{m}{2}\right\rceil$.  As in \Cref{ubound2}, chain sets beginning at the vertices $\{p_i\}_{i=1}^{2r}$ can force $\{v_{i,j}\}_{i=1,}^{2r}{}_{j=1}^r$ blue and will finish doing so at the conclusion of time-step $2r-1$.

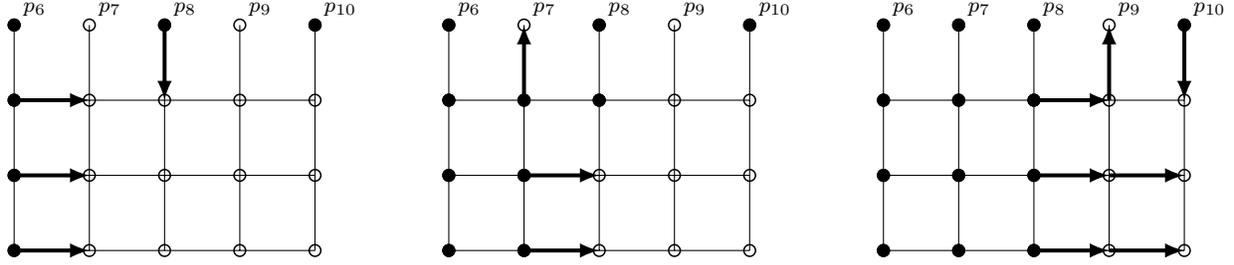
\begin{figure}[htbp]
    \centering
    \begin{minipage}[b]{0.3\textwidth}
        \centering
        \begin{tikzpicture}[scale=1] 
            \draw[black] (1,1) grid (5,3);
            \foreach \x in {1,...,5} {
                \foreach \y [evaluate=\y as \yy using {int(4 - \y)}] in {1,...,3}{
                    \draw[line width=0.2mm]  (\x,\y) circle (2.2pt); 
                }
            }

            \foreach \p in {1,...,5}{
                \draw [black] (\p, 4) -- (\p, 3);
                \draw[line width=0.2mm]  (\p, 4) circle (2.2pt);
                \pgfmathtruncatemacro{\i}{\p + 5}
                \node[anchor=south west, font=\footnotesize] at (\p,4) {$p_{\i}$}; 

                \ifnum\p=1
                    \fill (\p, 4) circle (2.5pt);
                    \foreach \x in {1,...,3} {
                        \fill (\p, \x) circle (2.5pt);   
                        \draw[->, >=latex, line width=1.5pt] (\p, \x) -- (\p+1, \x);
                    }
                \else
                
                    \ifnum\p=2
                    \else
                        \ifnum\p=3
                            \fill (\p, 4) circle (2.5pt);
                            \draw[->, >=latex, line width=1.5pt] (\p, 4) -- (\p, 3);
                        \else
                            \ifnum\p=4
                            \else   
                                \ifnum\p=5
                                    \fill (\p, 4) circle (2.5pt);
                                \fi
                            \fi
                        \fi
                    \fi
                \fi
            }

        \end{tikzpicture}
    \end{minipage}
    \hfill
    \begin{minipage}[b]{0.3\textwidth}
        \centering
        \begin{tikzpicture}[scale=1] 
            \draw[black] (1,1) grid (5,3);
            \foreach \x in {1,...,5} {
                \foreach \y [evaluate=\y as \yy using {int(4 - \y)}] in {1,...,3}{
                    \draw[line width = 0.2mm] (\x,\y) circle (2.2pt); 
                }
            }

            \foreach \p in {1,...,5}{
                \draw [black] (\p, 4) -- (\p, 3);
                \draw [line width=0.2mm] (\p, 4) circle (2.2pt);
                \pgfmathtruncatemacro{\i}{\p + 5}
                \node[anchor=south west, font=\footnotesize] at (\p,4) {$p_{\i}$}; 

                \ifnum\p=1
                    \fill (\p, 4) circle (2.5pt);
                    \foreach \x in {1,...,3} {
                        \fill (\p, \x) circle (2.5pt);   
                    }
                \else
                
                    \ifnum\p=2
                        \draw[->, >=latex, line width=1.5pt] (\p, 3) -- (\p, 4);
                        \foreach \x in {1,...,3} {
                            \fill (\p, \x) circle (2.5pt);   
                        }
                        \draw[->, >=latex, line width=1.5pt] (\p, 1) -- (\p+1, 1);
                        \draw[->, >=latex, line width=1.5pt] (\p, 2) -- (\p+1, 2);
                        
                    \else
                        \ifnum\p=3
                            \fill (\p, 4) circle (2.5pt);
                            \fill (\p, 3) circle (2.5pt);
                        \else
                            \ifnum\p=4
                            \else
                                \ifnum\p=5
                                    \fill (\p, 4) circle (2.5pt);
                                \fi
                            \fi
                        \fi
                    \fi
                \fi
            }

        \end{tikzpicture}
    \end{minipage}
    \hfill
    \begin{minipage}[b]{0.3\textwidth}
        \centering
                \begin{tikzpicture}[scale=1.0] 
            \draw[black] (1,1) grid (5,3);
            \foreach \x in {1,...,5} {
                \foreach \y [evaluate=\y as \yy using {int(4 - \y)}] in {1,...,3}{
                    \draw[line width=0.2mm]  (\x, \y) circle (2.2pt); 
                }
            }

            \foreach \p in {1,...,5}{
                \draw [black] (\p, 4) -- (\p, 3);
                \draw[line width=0.2mm]  (\p, 4) circle (2.2pt);
                \pgfmathtruncatemacro{\i}{\p + 5}
                \node[anchor=south west, font=\footnotesize] at (\p,4) {$p_{\i}$}; 

                \ifnum\p=1
                    \fill (\p, 4) circle (2.5pt);
                    \foreach \x in {1,...,3} {
                        \fill (\p, \x) circle (2.5pt);   
                    }
                \else
                
                    \ifnum\p=2
                        \foreach \x in {1,...,4} {
                            \fill (\p, \x) circle (2.5pt);   
                        }
                        
                    \else
                        \ifnum\p=3
                            \fill (\p, 4) circle (2.5pt);   
                            \foreach \x in {1,...,3} {
                                \fill (\p, \x) circle (2.5pt);  
                                \draw[->, >=latex, line width=1.5pt] (\p, \x) -- (\p+1, \x);
                            }
                        \else
                            \ifnum\p=4
                                \draw[->, >=latex, line width=1.5pt] (\p, 3) -- (\p, 4);
                                \draw[->, >=latex, line width=1.5pt] (\p, 2) -- (\p+1, 2);
                                \draw[->, >=latex, line width=1.5pt] (\p, 1) -- (\p+1, 1);
                                
                            \else   
                                \ifnum\p=5
                                    \fill (\p, 4) circle (2.5pt);
                                    \draw[->, >=latex, line width=1.5pt] (\p, 4) -- (\p, 3);

                                \fi
                            \fi
                        \fi
                    \fi
                \fi
            }

        \end{tikzpicture}
    \end{minipage}
    \caption{Time-steps $2r=6$ through $m-2r-1=9$ of $\mathcal F$ in $Wb(16,3)$}
    \label{fig:upperbound_3.8_middle}
\end{figure}

Now, during time-step $2r$ the forces $\{v_{2r,j}\rightarrow v_{2r+1,j}\}_{j=1}^{r}$ as well as $p_{2r+2}\rightarrow v_{2r+2,1}$ can occur.  At which point, the only white neighbor of $v_{2r+1,1}$ will be $p_{2r+1}$, so the force $v_{2r+1,1} \rightarrow p_{2r+1}$ as well as the set of forces $\{v_{2r+1,j}\rightarrow v_{2r+2,j}\}_{j=2}^r$ can occur during time-step $2r+1$.  Time-steps analogous to Time-steps $2r$ and $2r+1$ will alternate (with one additional time-step during which $\{v_{m-2r-1,j}\rightarrow v_{m-2r,j}\}_{j=1}^r$ in the case where $m$ is odd) until time-step $m-2r-1$, at which point $\{v_{i,j}\}_{i=1,}^{m-2r}{}_{j=1}^r \cup \{p_i\}_{i=1}^{m-2r} \subseteq E_{\mathcal F}^{[m-2r-1]}$.  Finally, during time-step $m-2r$, the sets of forces $\{v_{1,j}\rightarrow v_{m,j}\}_{j=1}^r$ and $\{v_{m-2r,j}\rightarrow v_{m-2r+1,j}\}_{j=1}^r$ can occur.

\begin{figure}[htbp]
    \centering
    \begin{minipage}[b]{0.8\textwidth} 
        \centering
        \begin{tikzpicture}[scale=1] 
            \draw[black] (1,1) grid (6,3);
            \foreach \x in {1,...,6} {
                \foreach \y [evaluate=\y as \yy using {int(4 - \y)}] in {1,...,3}{
                    \fill (\x,\y) circle (1.5pt); 
                    \pgfmathtruncatemacro{\i}{\x + 10}
                    \node[anchor=south west, font=\footnotesize] at (\x,\y) {$v_{\i,\yy}$}; 
                }
            }
            \foreach \p in {1,...,6}{
                \draw (\p, 4) -- (\p, 1);
                \fill (\p, 4) circle (1.5pt);

                \pgfmathtruncatemacro{\i}{\p + 10}
                \node[anchor=south west, font=\footnotesize] at (\p,4) {$p_{\i}$}; 
            }

        \end{tikzpicture}
        \caption{Subgraph $H'$ of $Wb(16,3)$ discussed in \Cref{ubound3}}
        \label{fig:upperbound_3.7}
    \end{minipage}
\end{figure}
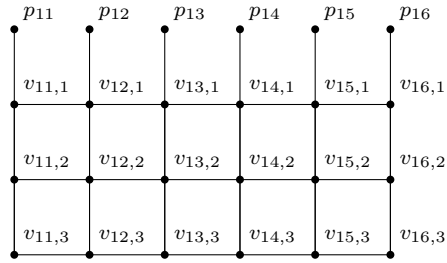

\begin{figure}[htbp]
    \centering
    \begin{minipage}[b]{0.45\textwidth}
        \centering
        \begin{tikzpicture}[scale=1.2] 
            \draw[black] (1,1) grid (6,3);
            \foreach \x in {1,...,6} {
                \foreach \y [evaluate=\y as \yy using {int(4 - \y)}] in {1,...,3}{
                    \draw[line width=0.2mm] (\x,\y) circle (2.4pt); 
                }
            }

            \foreach \p in {1,...,6}{
                \draw [black] (\p, 4) -- (\p, 3);
                \draw[line width=0.2mm] (\p, 4) circle (2.4pt);
                \pgfmathtruncatemacro{\i}{\p + 10}
                \node[anchor=south west, font=\footnotesize] at (\p,4) {$p_{\i}$}; 

                \ifnum\p=1
                    \fill (\p, 3) circle (2.5pt); 
                    \foreach \x in {1,...,2} {
                        \fill (\p, \x) circle (2.5pt);   
                        \draw[->, >=latex, line width=1.5pt] (\p, \x) -- (\p+1, \x);
                    }
                \else
                
                    \ifnum\p=2
                        \draw[->, >=latex, line width=1.5pt] (\p, 1) -- (\p+1, 1);
                    \else
                        \ifnum\p=3
                        \else
                            \ifnum\p=4
                            \else   
                                \ifnum\p=5
                                    \draw[->, >=latex, line width=1.5pt] (\p, 1) -- (\p-1, 1);
                                \else   
                                    \ifnum\p=6
                                        \fill (\p, 3) circle (2.5pt); 
                                        \foreach \x in {1,...,2} {
                                        \fill (\p, \x) circle (2.5pt);   
                                        \draw[->, >=latex, line width=1.5pt] (\p, \x) -- (\p-1, \x);
                    }
                                    \fi
                                \fi
                            \fi
                        \fi
                    \fi
                \fi
            }

        \end{tikzpicture}
    \end{minipage}
    \hfill
    \begin{minipage}[b]{0.45\textwidth}
        \centering
        \begin{tikzpicture}[scale=1.2] 
            \draw[black] (1,1) grid (6,3);
            \foreach \x in {1,...,6} {
                \foreach \y [evaluate=\y as \yy using {int(4 - \y)}] in {1,...,3}{
                    \draw[line width=0.2mm] (\x,\y) circle (2.4pt); 
                }
            }

            \foreach \p in {1,...,6}{
                \draw [black] (\p, 4) -- (\p, 3);
                \draw[line width=0.2mm] (\p, 4) circle (2.4pt);
                \pgfmathtruncatemacro{\i}{\p + 10}
                \node[anchor=south west, font=\footnotesize] at (\p,4) {$p_{\i}$}; 

                \ifnum\p=1
                    \foreach \x in {1,...,3} {
                        \fill (\p, \x) circle (2.5pt);   
                    }
                    \draw[->, >=latex, line width=1.5pt] (\p, 3) -- (\p, 4);
                \else
                
                    \ifnum\p=2
                        \foreach \x in {1,...,2} {
                            \fill (\p, \x) circle (2.5pt);   
                        }
                        \foreach \x in {2,...,3} {
                            \draw[->, >=latex, line width=1.5pt] (\p, \x) -- (\p, \x+1);
                        }
                        
                    \else
                        \ifnum\p=3
                            \fill (\p, 1) circle (2.5pt);
                            \foreach \x in {1,...,3} {
                                \draw[->, >=latex, line width=1.5pt] (\p, \x) -- (\p, \x+1);
                            }   
                        \else
                            \ifnum\p=4
                                \fill (\p, 1) circle (2.5pt);
                                \foreach \x in {1,...,3} {
                                    \draw[->, >=latex, line width=1.5pt] (\p, \x) -- (\p, \x+1);
                                } 
                            \else
                                \ifnum\p=5
                                    \foreach \x in {1,...,2} {
                                        \fill (\p, \x) circle (2.5pt);   
                                    }
                                    \foreach \x in {2,...,3} {
                                          \draw[->, >=latex, line width=1.5pt] (\p, \x) -- (\p, \x+1);
                                    }
                                \else
                                    \ifnum\p=6
                                        \foreach \x in {1,...,3} {
                                            \fill (\p, \x) circle (2.5pt);   
                                        }
                                        \draw[->, >=latex, line width=1.5pt] (\p, 3) -- (\p, 4);
                                    \fi
                                   
                                \fi
                            \fi
                        \fi
                    \fi
                \fi
            }

        \end{tikzpicture}
    \end{minipage}
    \caption{The last $2r-1=5$ time-steps of $\mathcal F$ in $Wb(16,3)$}
\end{figure}
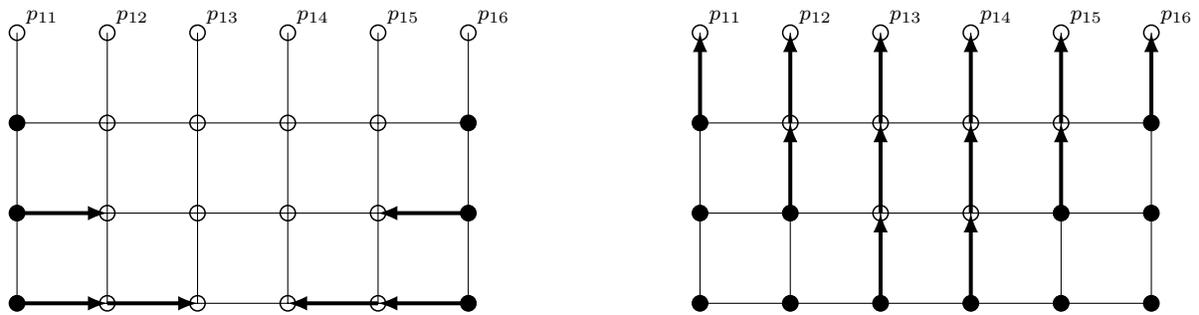

Next, in a way similar to \Cref{ubound2}, define $H$ and $H'$ to be $Wb(m,r)\left[\{v_{i,j}\}_{i=1,}^{2r}{}_{j=1}^{r} \cup \{p_i\}_{i=1}^{2r}\right]$ and $Wb(m,r)[\{v_{i,j}\}_{i=m-2r+1,}^m{}_{j=1}^r \cup \{p_i\}_{i=m-2r+1}^m]$ respectively.  Note that again $H$ and $H'$ are isomorphic by some isomorphism $\sigma$, and as before the fact that $\Term(\mathcal F|_H)$ is a zero forcing set of $H$ implies that $\sigma(\Term(\mathcal F|_H))$ is a zero forcing set of $H'$.  Since $\sigma(\Term(\mathcal F|_H))=\{v_{m,j}, v_{m-2r+1,j}\}_{j=1}^r \subseteq E_{\mathcal F}^{[m-2r]}$, it follows that $E_{\mathcal F}^{[m-2r]}$ contains a zero forcing set of $H'$.  Finally, since there are no white vertices outside of $H'$, $B$ is a zero forcing set of $Wb(m,r)$ and $\Z(Wb(m,r))\leq \left\lceil\frac{m}{2}\right\rceil$.    
\end{proof}

Combining Lemmas \ref{weblower}, \ref{ubound1}, \ref{ubound2}, and \ref{ubound3}, one obtains the following theorem.

\begin{theorem}
$\Z(Wb(m,r)) = \max\left\{\left\lceil\frac{m}{2}\right\rceil, \min\{m,2r\}\right\}$, or equivalently
\[\Z(Wb(m,r))=
\begin{cases}
m & \text{if } m \leq 2r,\\
2r & \text{if } \lceil \frac{m}{2}\rceil < 2r < m, \\
\lceil \frac{m}{2}\rceil & \text{if } 2r \leq \lceil \frac{m}{2}\rceil.
\end{cases}\]
\end{theorem}

\section{Conclusion}

This wraps up our exploration of peony graphs and web graphs.  In this paper we have seen that recent discoveries concerning the interactions of graph substructures and the zero forcing numbers of graphs can provide not only new abstract concepts but also concrete tools for determining the zero forcing numbers of relatively complex graphs classes.  Since new graph substructures related to zero forcing are an active topic of research in this area, it seems that further exploration of the applications of these new substructures warrants additional study.

\section*{Acknowledgements}

The research of Kanno Mizozoe was supported by the Trinity College Summer Research Program.

\bibliographystyle{plain}
\bibliography{peony}

\end{document}